\documentclass[12pt]{amsart}
\usepackage{amsmath, palatino}
\usepackage{amsthm, amsfonts, dsfont}
\usepackage{enumerate}
\usepackage{amssymb}
\usepackage[utf8]{inputenc}
\usepackage[all,cmtip]{xy}

\theoremstyle{plain}
\newtheorem{lemma}{Lemma}[section]
\newtheorem{theorem}[lemma]{Theorem}
\newtheorem{corollary}[lemma]{Corollary}
\newtheorem{proposition}[lemma]{Proposition}
\newtheorem{definition}[lemma]{Definition}
\newtheorem*{proposition*}{Proposition}
\newtheorem*{definition*}{Definition}
\newtheorem*{claim*}{Claim}

\newtheorem*{notation*}{Notation}

\newtheorem{remark}[lemma]{Remark}
\newtheorem{example}[lemma]{Example}

\newcommand{\Cu}{\mathrm{Cu}}

\newcommand{\supp}{\mathrm{supp}}
\newcommand{\C}{\mathrm{C}}
\begin{document}

\title{The Cuntz semigroup of continuous fields}
\date{\today}
\author{Ramon Antoine}
\author{Joan Bosa}
\author{Francesc Perera}
\address{Departament de Matem\`atiques, Universitat Aut\`onoma de Barcelona, 08193 Bellaterra, Barcelona, Spain}\email{ramon@mat.uab.cat, jbosa@mat.uab.cat, perera@mat.uab.cat}

\begin{abstract}
In this paper we describe the Cuntz semigroup of continuous fields of C$^*$-algebras over one dimensional spaces whose fibers have stable rank one and trivial $K_1$ for each closed, two-sided ideal. This is done in terms of the semigroup of global sections on a certain topological space built out of the Cuntz semigroups of the fibers of the continuous field. When the fibers have furthermore real rank zero, and taking into account the action of the space, our description yields that the Cuntz semigroup is a classifying invariant if and only if so is the sheaf induced by the Murray-von Neumann semigroup.
%is, as a functor, equivalent to the sheaf induced by the Murray-von Neuman semigroup.  
\end{abstract}

\maketitle

\section*{Introduction}
The Cuntz semigroup has become a popular object in recent years, mainly due to its connection with the classification program of unital, simple, separable and nuclear C$^*$-algebras by means of the Elliott invariant. It can be thought of as a functorial invariant $\Cu(\_)$ from the category of C$^*$-algebras to a category of semigroups, termed $\Cu$, that has certain continuity properties. The Elliott invariant is also functorial and consists of ordered topological K-Theory, the trace simplex and the pairing between K-Theory and traces; it is customarily denoted by $\mathrm{Ell}(\_)$. The Elliott conjecture asserts that an isomorphism between the invariants $\mathrm{Ell}(A)$ and $\mathrm{Ell}(B)$ of C$^*$-algebras $A$ and $B$ may be lifted to a $^*$-isomorphism between $A$ and $B$. Although the conjecture fails in general (see \cite{ro}, \cite{toms}), it has been verified for large classes of  C$^*$-algebras, all of which happen to absorb the Jiang-Su algebra $\mathcal Z$ tensorially. For these algebras, the Cuntz semigroup can be recovered functorially from the Elliott invariant (see \cite{bpt}). More recently, it was shown further in \cite{ADPS} (see also \cite{tiku}) that the Elliott invariant can be recovered from the Cuntz semigroup after tensoring with C$(\mathbb{T})$, and thus $\mathrm{Ell}(\_)$ and $\Cu(\mathrm{C}(\mathbb{T},\_))$ define equivalent functors.

In the non-simple case, the Cuntz semigroup has already been used successfully to classify certain classes of C$^*$-algebras, such as AI algebras (\cite{ciupercaelliott}), inductive limits of one dimensional non commutative CW-complexes with trivial $K_1$  (\cite{robadv}) or inductive limits of certain continuous-trace C$^*$-algebras (\cite{ciupercaelliottsantiago}), among others. Another class of (non-simple) algebras for which classification results have been obtained are continuous fields over $[0,1]$ of either Kirchberg algebras (with certain torsion freeness assumptions on their K-Theory) or AF algebras (\cite{DE,DEN}). In this situation, the classifying invariant consists of the sheaf of groups naturally induced by K-Theory. In the stably finite case, it is natural to ask whether the Cuntz semigroup of the continuous field captures, on its own, all the information of the K-Theory sheaf. This is one of the main objectives pursued in this article, and we are able to settle the question positively for a wide class of continuous fields.
%This would both unify a whole swath of invariants into a single one (?) and suggest its future use in further classification results. This is one of the objectives pursued in this article. O be, this is a theme that drives our investigations in this article.

In order to achieve our aim, we need techniques that allow to compute the Cuntz semigroup of continuous fields. In the case of algebras of the form $\mathrm{C}_0(X,A)$, for a locally compact Hausdorff space $X$, this was carried out in \cite{tiku} whenever $A$ is a simple, unital, non-type I ASH-algebra with slow dimension growth. For a not necessarily simple algebra $A$ of stable rank one with vanishing $K_1$ for each closed, two-sided ideal, and for a one dimensional locally compact Hausdorff space, one of the main results in \cite{APS} proves that there is an isomorphism between $\Cu\mathrm{(C}_0(X,A))$ and $\mathrm{Lsc}(X,\Cu(A))$. The latter is a semigroup of $\Cu(A)$-valued lower semicontinuous functions (see below for the precise definitions), and the isomorphism is given by point evaluation of (representatives of) Cuntz classes. The more general situation of $\mathrm{C}(X)$-algebras was also a theme developed in \cite{APS}, where spaces of dimension at most one and $\mathrm{C}(X)$-algebras whose fibers have stable rank one and vanishing $K_1$ for each closed, two-sided ideal were considered. In this paper, we shall refer to this class as \emph{$\mathrm{C}(X)$-algebras without $K_1$ obstructions}. For such spaces, the Cuntz semigroup of a $\mathrm{C}(X)$-algebra $A$ with no $K_1$ obstructions embeds into the product $\prod_{x\in X}\Cu(A_x)$. (This was shown in \cite{APS} for $X=[0,1]$, and we prove the one dimensional case here, based on the pullback construction carried out also in \cite{APS}.) 

The remaining problem of identifying the image of $\Cu(A)$ in $\prod_{x\in X}\Cu(A_x)$ for $\mathrm{C}(X)$-algebras $A$ in the said class leads to the analysis of the natural map $F_{\Cu(A)}:=\sqcup_{x\in X}\Cu(A_x)\to X$ and its sections. This is motivated by the fact that the $\Cu$ functor induces a presheaf $\Cu_A$ on $X$ that assigns, to each closed set $U$ of $X$, the semigroup $\Cu(A(U))$.  Hence we may expect to relate the Cuntz semigroup with the semigroup of continuous sections of an \'etale bundle. In the case of the presheaf defined by K-Theory, and for some continuous fields over $[0,1]$, this was considered in \cite{DE}. We show that, for a one dimensional space $X$ and a $\mathrm{C}(X)$-algebra with no $K_1$ obstructions, the presheaf $\Cu_A$ is in fact a sheaf. In order to recover $\Cu_A$ from the sheaf of continuous sections of the map $F_{\Cu(A)}\to X$, we need to break away from the standard approach (of, e.g. \cite{WE}) and consider a topological structure on $F_{\Cu(A)}$ that takes into account continuity properties of the objects in the category $\Cu$. Thus we develop a more abstract analysis of $\Cu$-valued sheaves that follows, in part, the spirit of \cite{APS}. This culminates in Theorem \ref{isomorfisme}, which allows to recover the Cuntz semigroup of a $\mathrm{C}(X)$-algebra $A$ with no $K_1$ obstructions over a one dimensional space $X$ as the semigroup of global sections on $F_{\Cu(A)}$. 

To conclude the paper, we apply the previous result in a crucial way to prove that, for one dimensional spaces and $\mathrm{C}(X)$-algebras with no $K_1$ obstructions whose fibers have real rank zero, the Cuntz semigroup and the $K$-theoretical sheaf defined by the Murray-von Neumann semigroup carry the same information. (This sheaf is defined, for a $\mathrm{C}(X)$-algebra $A$, as $\mathbb{V}_A(U)=V(A(U))$ whenever $U$ is a closed subset of $X$.) A key ingredient here is that the natural module structure of a $\mathrm{C}(X)$-algebra $A$ equips $\Cu(A)$ with an enriched structure via an action of $\Cu(\mathrm{C}(X))$. Thus, more precisely, we show that if $A$ and $B$ are two such $\mathrm{C}(X)$-algebras, then there is an action preserving semigroup isomorphism $\Cu(A)\cong \Cu(B)$ if, and only if, the sheaves $\mathbb{V}_A(\_)$ and $\mathbb{V}_B(\_)$ are isomorphic (Theorem \ref{V(A)}). 

\section{Preliminaries}
\subsection{Cuntz Semigroup}
%The Cuntz semigroup of a C$^*$-algebra $A$ is built out of equivalence classes of positive elements of the stabilization $A\otimes\mathcal{K}$. 
Let $A$ be a C$^*$-algebra, and let $a$, $b\in A_+$. We say that $a$ is \emph{Cuntz subequivalent} to $b$, in symbols $a\preceq b$, provided there is a sequence $(x_n)$ in $A$ such that $x_nbx_n^*$ converges to $a$ in norm. We say that $a$ and $b$ are \emph{Cuntz equivalent} if $a\preceq b$ and $b\preceq a$, and in this case we write $a\sim b$.

The Cuntz semigroup is defined as the quotient set $\Cu(A)=(A\otimes \mathcal{K})_+/\!\!\!\sim$, and its elements are denoted by $[a]$, for $a\in (A\otimes\mathcal{K})_+$. This set becomes an ordered semigroup, with order induced by Cuntz subequivalence and addition given by $[a]+[b]=[\Theta\left(\begin{smallmatrix} a & 0 \\ 0 & b \end{smallmatrix}\right)]$, where $\Theta\colon M_2(A\otimes\mathcal{K})\to A\otimes \mathcal{K}$ is any inner isomorphism.
The following summarizes some technical properties of Cuntz subequivalence that will be used in the sequel.

\begin{proposition}\label{here}{\rm (\cite{R1}, \cite{RK})}
 Let $A$ be a C*-algebra, and $a,b\in A_{+}$. The following are equivalent:
\begin{enumerate}[\rm(i)]
 \item $a\preceq b$.
  \item For all $\epsilon> 0$, $(a-\epsilon)_{+}\preceq b$.
\item For all $\epsilon> 0$, there exists $\delta> 0$ such that $(a-\epsilon)_{+}\preceq (b-\delta)_{+}$.

Furthermore, if $A$ is stable, these conditions are equivalent to
\item For every $\epsilon> 0$ there is a unitary $u\in U(\tilde{A})$ such that $u(a-\epsilon)_{+}u^{*}\in Her(b)$.
\end{enumerate}
\end{proposition}

The structure of the Cuntz semigroup is richer than just being an ordered semigroup, as it belongs to a category with certain continuity properties.
Recall that, in an ordered semigroup $S$, an element $s$ is said to be \emph{compactly contained} in $t$,
denoted $s\ll t$, if whenever $t\leq \sup_{n} z_{n}$ for some increasing sequence $(z_{n})$
with supremum in $S$, there exists $m$ such that $s\leq z_m$. An element $s$ is said to be \emph{compact} if $s\ll s$. A
sequence $(s_{n})$ such that $s_{n}\ll s_{n+1}$ is termed \emph{rapidly increasing}. The following theorem summarizes some structural properties of the Cuntz semigroup.

\begin{theorem}{\rm (\cite{CEI})}
 Let $A$ be a C$^*$-algebra. Then:
\begin{enumerate}[\rm(i)]
\item Every increasing sequence in $\Cu(A)$ has a supremum in $\Cu (A)$.
\item Every element in $\Cu(A)$ is the supremum of a rapidly increasing sequence.
\item The operation of taking suprema and $\ll$ are compatible with addition.
\end{enumerate}
\end{theorem}

This allows to define a category $\Cu$ whose objects are
ordered semigroups of positive elements satisfying
conditions (i)-(iii) above. (Morphisms in this category are those semigroup maps that preserve all the structure.) We say that a semigroup $S$ in the category $\Cu$ is \emph{countably based} if there exists a countable
subset $X$ that is dense in $S$, meaning that every element of $S$ is the supremum of a rapidly increasing
sequence of elements in $X$. It was observed in \cite{APS} (see also \cite{Robert}) that $\Cu(A)$ is countably based for any separable C$^*$-algebra $A$.

As shown in \cite{CEI}, the category $\Cu$ is closed under countable inductive limits (in fact, it was also shown that $\Cu$ defines a sequentially continuous functor from the category of C$^*$-algebras to $\Cu$). A useful description of the inductive limit is available below.

\begin{proposition}\label{lionel}{\rm (\cite{CEI}, cf. \cite{BRTTW})}
Let $(S_{i},\alpha_{i,j})_{i,j\in \mathbb{N}}$ be an inductive system in the category $\Cu$.
Then  $(S,\alpha_{i,\infty})$ is the inductive limit of this system if
\begin{enumerate}[{\rm (i)}]
 \item The set $\bigcup_i \alpha_{i,\infty}(S_{i})$ is dense in $S$.
\item For any $x,y\in S_i$ such that $\alpha_{i,\infty}(x)\leq \alpha_{i,\infty}(y)$ and $x'\ll x$
there is $j$ such that $\alpha_{i,j}(x')\leq \alpha_{i,j}(y)$. 
\end{enumerate}
\end{proposition}

If $S$ is a semigroup in $\Cu$, an \emph{order-ideal} $I$ of $S$ is a subsemigroup which is order-hereditary (that is, $x\in I$ whenever $x\leq y$ and $y\in I$) and that further contains all suprema of increasing sequences in $I$. For example, if $A$ is a C$^*$-algebra and $I$ is a closed, two-sided ideal, then $\Cu(I)$ is naturally an order-ideal of $\Cu(A)$. Given an order-ideal $I$ of $S$ as before, define a congruence relation on $S$ by $s\sim t$ if $s\leq t+z$ and $t\leq s+w$ for some $z$, $w\in I$, and put $S/I=S/\!\!\sim$, which is an ordered semigroup with addition $[s]+[t]=[s+t]$ and order given by $[s]\leq [t]$ if $s\leq t+z$ for some $z\in I$. It is not hard to verify that $S/I\in\Cu$. For a C$^*$-algebra $A$ and a closed ideal $I$, it was proved in \cite{crobsant} that $\Cu(A/I)\cong\Cu(A)/\Cu(I)$ (where the isomorphism is induced by the natural quotient map).

\subsection{C(X)-algebras}

Let $X$ be a compact Hausdorff space. A \emph{$\mathrm{C}(X)$-algebra} is a C$^*$-algebra $A$ together with a unital $^*$-homomorphism $\theta\colon \mathrm{C}(X)\to Z(\mathcal{M}(A))$, where $\mathcal{M}(A)$ is the multiplier algebra of $A$. The map $\theta$ is usually referred to as the \emph{structure map}. We write $fa$ instead of $\theta(f)a$ where $f\in \mathrm{C}(X)$ and $a\in A$.

If $Y\subseteq X$ is a closed set, let $A(Y)=A/C_0(X\smallsetminus Y)A$, which also becomes a $\mathrm{C}(X)$-algebra.
The quotient map is denoted by $\pi_Y\colon A\to A(Y) $, and if $Z$ is a closed
subset of $Y$ we have a natural restriction map $\pi_{Z}^Y\colon A(Y)\to A(Z)$. Notice that $\pi_Z=\pi_{Z}^Y\circ \pi_Y$.
If $Y$ reduces to a point $x$, we write $A_x$ instead of $A(\{x\})$ and we denote by $\pi_x$ the quotient map. The C$^*$-algebra $A_x$ is called the \emph{fiber} of $A$ at $x$ and the image of $\pi_x(a)\in A_x$ will be denoted by 
$a(x)$.

Given a $\mathrm{C}(X)$-algebra $A$ and $a\in A$, the map $x\mapsto \|a(x)\|$ is upper semicontinuous (see \cite{Blan}). If this map is actually continuous for every $a\in A$, then we say that $A$ is a \emph{continuous field} (or also a C$^*$-bundle, see \cite{nilsen,Blan}). For a continuous field $A$, a useful criterion to determine when an element $(a_x)\in \prod_{x\in X}A_x$ comes from an element of $A$ is the following: given $\epsilon>0$ and $x\in X$, if there is $b\in A$ and a neighborhood $V$ of $x$ such that $\|b(y)-a_y\|<\epsilon$ for $y\in V$, then there is $a\in A$ such that $a(x)=a_x$ for all $x$ (see \cite[Definition 10.3.1]{Dixmier}).

It was proved in \cite[Lemma 1.5]{APS} that, if $A$ is a $\mathrm{C}(X)$-algebra, then this is also the case for $A\otimes\mathcal{K}$ and, in fact, for any closed set $Y$ of $X$, there is a $^*$-isomorphism
\[
\varphi_Y\colon(A\otimes\mathcal{K})(Y)\to A(Y)\otimes\mathcal{K}
\] 
such that $\varphi_Y\circ \pi_Y '=\pi_Y\otimes 1_{\mathcal{K}}$, where
$\pi_Y\colon A\to A(Y)$ and $\pi'_Y\colon A\otimes\mathcal{K}\to (A\otimes\mathcal{K})(Y)$. This yields, in particular, that $(A\otimes\mathcal{K})(x)\cong A_x\otimes\mathcal{K}$ for any $x\in X$, with $(a\otimes k)(x)\mapsto a(x)\otimes k$.

Using this observation, the map induced at the level of Cuntz semigroups $\Cu(A)\to \Cu(A_x)$ can be 
viewed as $[a]\mapsto [\pi_x(a)]$. Similarly, if $Y$ is closed in $X$, the map $\pi_Y$ induces $\Cu(A)\to \Cu(A(Y))$, that can be thought of as
$[a]\mapsto [\pi_Y(a)]$. Thus, when computing the Cuntz semigroup of a $\mathrm{C}(X)$-algebra $A$, we may and will assume that $A$, $A_x$ and $A(Y)$ are stable.

\section{Sheaves of semigroups and continuous sections}

Our aim in this section is to relate the Cuntz semigroup of a $\mathrm{C}(X)$-algebra $A$ with the semigroup of continuous sections of a certain topological space, which is built out of the information on the fibers. We will first define what is meant by a presheaf of semigroups on a topological space, along the lines of \cite{WE}, with some modifications.

Let $X$ be a topological space. As a blanket assumption, we shall assume that $X$ is always compact, Hausdorff and second countable, therefore metrizable. 
Denote by $\mathcal{V}_X$ the category of all closed subsets of $X$ with non-empty interior, with the morphisms given by inclusion.

A \emph{presheaf} over $X$ is a contravariant functor
\[
\mathcal{S}\colon\mathcal{V}_X\to \mathcal{C}\,
\]
where $\mathcal{C}$ is a subcategory of the category of sets which is closed under sequential inductive limits. Thus, it consists of an assignment, for each $V\in\mathcal{V}_X$ of
an object $\mathcal{S}(V)$ in $\mathcal{C}$ and a collection of maps (referred to as \emph{restriction homomorphisms})
$\pi^{V'}_V\colon\mathcal{S}(V')\to \mathcal{S}(V)$ whenever $V\subseteq V'$ in
$\mathcal{V}_X$. We of course require that these maps satisfy $\pi^V_V=\mathrm{id}_V$ and $\pi^U_W=\pi^V_W\pi^U_V$ if $W\subseteq
V\subseteq U$.

Let $V,V'\in\mathcal{V}_X$ be such that $V\cap V'\in\mathcal{V}_X$. 
A presheaf is called a \emph{sheaf} if the map 
\[\pi_{V}^{V\cup V'}\times\pi_{V'}^{V\cup V'}\colon 
\mathcal{S}(V\cup V')\to \{(f,g)\in \mathcal{S}(V)\times\mathcal{S}(V')\mid
\pi_{V\cap V'}^V(f)=\pi_{V\cap V}^{V'}(g)\},
\]
is bijective.

A presheaf (respectively a sheaf) is \emph{continuous} if for any decreasing sequence of
closed subsets $(V_i)_{i=1}^\infty$ whose intersection $\cap_{i=1}^\infty V_i=V$
belongs to $\mathcal{V}_X$, the limit $\lim \mathcal{S}(V_i)$ is
isomorphic to $\mathcal{S}(V)$. 

Consider a presheaf $\mathcal{S}$ over $X$. For any $x\in X$, define the \emph{fiber} (or also \emph{stalk}) of $\mathcal{S}$ at $x$ as  
\[
S_x:=\lim_{x\in \mathring{V}}\mathcal{S}(V)\,,
\]
with respect to the restriction maps.

We shall be exclusively concerned with continuous presheaves (or sheaves) $\mathcal S$ with target values in the category $\mathrm{Sg}$ of semigroups, in which case we will say that $\mathcal S$ is a (pre)sheaf of semigroups. As a general notation, we will use $S$ to denote the semigroup $\mathcal{S}(X)$. We will also denote $\pi_x\colon S\to S_x$ the natural map from $S$ to the fiber $S_x$, as well as $\pi_U\colon S\to S(U)$ rather than $\pi_U^X$. 

Our main motivation for considering presheaves of semigroups stems from the study of $\mathrm{C}(X)$-algebras. Indeed, as it is easy to verify, given a $\mathrm{C}(X)$-algebra $A$, the assignments 
\[
\begin{array}{cccc} \Cu_A\colon &   \mathcal{V}_X&\to    &\Cu \\
				&  U             &\mapsto&\Cu(A(U)) \end{array} \,\,\,\, \,\,\,\,\text{ and } \,\,\,\, \,\,\,\, \begin{array}{cccc}\mathbb{V}_A\colon&   \mathcal{V}_X&\to& \mathrm{Sg} \\  
																  &  U&           \mapsto &V(A(U))\end{array}
\]
define continuous presheaves of semigroups. If $U\subseteq V$, the restriction maps $\pi_V^ U\colon A(U)\to A(V)$ and the limit maps $\pi_x\colon A\to A_x$ define,  by functoriality, semigroup maps $\Cu(\pi_V^ U)$ and $\Cu(\pi_x)$ in the case of the Cuntz semigroup, and likewise in the case for the semigroup of projections. For ease of notation, and unless confusion may arise, we shall still denote these maps by $\pi_V^ U$ and $\pi_x$.

We will say that a (pre)sheaf is \emph{surjective} provided all the restriction maps are surjective. This is clearly the case for the presheaf $\Cu_A$ for a general $\mathrm{C}(X)$-algebra $A$, and also for $\mathbb{V}_A$ if $A$ has real rank zero (which is a rather restrictive hypothesis, see e.g. \cite{pasnicu2} and \cite{pasnicu1}). As we shall see in the sequel, 
%still
$\Cu_A$ and $\mathbb{V}_A$ determine each other under milder assumptions.

Most of the discussion in this and the subsequent section will consider surjective (pre)\- sheaves of semigroups $\mathcal{S}\colon \mathcal{V}_X\to\Cu$, and we will need to develop a somewhat abstract approach on how to recover the information of the sheaf from the sheaf of sections of a bundle $F_S\to X$, where $F_S$ stands for the disjoint union of all the fibers (see\cite{WE}). This is classically done by endowing $F_S$ with a topological structure that glues together the fibers (which are computed as algebraic limits in the category of sets). One of the main difficulties here resides in the fact that the inductive limit in $\Cu$ is not the algebraic inductive limit, even in the case of the fiber of a surjective presheaf. We illustrate this situation below with an easy example. For a semigroup $S$ in $\Cu$ and a compact Hausdorff space $X$, we shall denote by $\mathrm{Lsc}(X, S)$ the set of those functions $f\colon X\to S$ such that $\{t\in X\mid f(t)\gg s\}$ is open in $X$ for all $s\in S$. If $X$ is finite dimensional, it was shown in \cite{APS} that $\mathrm{Lsc}(X,S)\in \Cu$.

\begin{example}
 \rm Let $A=\mathrm{C}([0,1],M_n(\mathbb{C}))$, where $n\geq 2$. We know that $\Cu(A)\cong \rm{Lsc}([0,1],\overline{\mathbb{N}})$, where $\overline{\mathbb{N}}=\mathbb{N}\cup\{\infty\}$ (see, e.g. \cite{Robert0}).

Now, let
$\{U_{m}=[\frac{1}{2}-\frac{1}{m},\frac{1}{2}+\frac{1}{m}]\}_{m\geq 2}$, which is a sequence of decreasing closed subsets of $[0,1]$ whose intersection is $\{1/2\}$. It is easy to verify that, in $\Cu$,
$$\lim_{\rightarrow}\Cu(A(U_{n}))=\lim_{\rightarrow}\mathrm{Lsc}(U_{n},\overline{\mathbb{N}})=\Cu(A(1/2))=\overline{\mathbb{N}}.$$

However, the computation of the direct limit above in the category of semigroups yields
$\{(a,b,c)\in\mathbb{N}^3\mid b\leq a,c\}$.
\end{example}

For a surjective continuous presheaf $\mathcal{S}\colon\mathcal{V}_X\to\Cu$, let $F_S:=\sqcup_{x\in
X}S_x$,  where $S_x=\lim_{x\in
\mathring{V}}S(V)$, and define $\pi\colon F_S\to X$ by $\pi(s)=x$ if $s\in S_x$.

We define a \emph{section} of $F_S$ as a function $f\colon X \to F_S$ such
that $f(x)\in S_x$. We equip the set of sections with pointwise addition and order, so this set 
becomes an ordered semigroup. Notice also that the set of sections is closed under pointwise
suprema of increasing sequences.

Any element $s\in S$ induces a section $\hat{s}$, which is defined by $\hat{s}(x)=\pi_x(s)\in S_x$ and will be referred to as the \emph{section induced } by $s$.  

\begin{lemma}\label{limit}
Let $\mathcal{S}\colon \mathcal{V}_X\to\Cu$ be a presheaf on $X$, and let $s$, $r\in S$.
\begin{enumerate}[{\rm (i)}]
\item If $\hat{s}(x)\leq
\hat{r}(x)$ for some $x\in X$ then, for each $s'\ll s$ in $S$ there is a
closed set $V$ with
$x\in\mathring{V}$ such that $\pi_{V}(s')\leq\pi_V(r)$. In particular,
$\hat{s'}(y)\leq\hat{r}(y)$ for all $y\in V$.
\item If, further, $\mathcal{S}$ is a sheaf, $U$ is a closed subset of $X$, and
$\hat{s}(x)\leq \hat{r}(x)$ for all $x\in U$, then for each $s'\ll s$ there is a
closed set $W$ of $X$ with $U\subset\mathring{W}$ and $\pi_W(s')\leq\pi_W(r)$.
\end{enumerate}
\end{lemma}

\begin{proof}
(i): Recall that $S_x=\lim \mathcal{S}(V_n)$, where $(V_n)$ is a decreasing
sequence of
closed sets whose intersection is $x$ (we may take $V_1=X$, and
$x\in\mathring{V_n}$ for all $n$). Then, by Proposition \ref{lionel}, two
elements $s,r$ satisfy $\hat{s}(x)=\pi_x(s)\leq\pi_x(r)=\hat{r}(x)$ in
$S_x$ if and only if for all $s'\ll s$ there exists $j\geq 1$ such that
$\pi_{V_j}(s')\leq \pi_{V_j}(r)$, and in particular $\hat{s'}(y)\leq
\hat{r}(y)$ for all $y$ in $V_j$.

(ii): Assume now that $\mathcal{S}$ is a sheaf, and take $s'\ll s$. Apply (i) to
each $x\in U$, so that we can find $U_x$ with $x\in \mathring{U}_x$ such that
$\pi_{U_x}(s')\leq\pi_{U_x}(r)$. By compactness of $U$, there are a finite
number $U_{x_1},\ldots,U_{x_n}$ whose interiors cover $U$. Put $W=\cup_i
U_{x_i}$. As $\mathcal{S}$ is a sheaf and $\pi_{U_{x_i}}(s')\leq
\pi_{U_{x_i}}(r)$ for all $i$, it follows that $\pi_W(s')\leq \pi_W(r)$.
\end{proof}

Following Lemma \ref{limit}, our aim is to define a topology in $F_{\mathcal{S}}$ for which the induced sections will be continuous. Instead of abstractly considering the final topology 
generated by the induced sections, we define a particular topology which will satisfy our needs. Given $U$ an open set in $X$ and $s\in S$, put 
\[\ U_{s}^{\gg}=\{a_x \in F_\mathcal{S}\mid \hat{s'}(x)\ll a_x \text{ for some } x\in U \text{ and some } s\ll s'  \}\,,\]
and  equip $F_{\mathcal{S}}$ with the topology generated by these sets. 

Now consider an induced section $\hat s$ for some $s\in S$, and an open set of the form $U_{r}^{\gg}$ for some $r\in S$ and $U\subseteq X$. 
Suppose  $x\in \hat s^{-1}(U_r^\gg)$. Note that $x\in U$ and that $\hat{s}(x)\gg \hat{s'}(x)$ for some $s'\gg r$. Using that
$s'=\sup(s'_n)$ for a rapidly increasing sequence $(s'_n)$, there exists
$n_0$ such that $r\ll s'_{n_{0}} \ll s' $. Hence, by Lemma \ref{limit}, there is a closed set $V$ such that $x\in\mathring{V}$ and
$\hat{s'}_{n_0}(y)\ll \hat{s}(y)$ for all $y$ in $V$. Thus, $x\in U\cap\mathring{V}\subseteq \hat{s}^{-1}(U_{r}^{\gg})$, proving that $\hat s^{-1}(U_r^\gg)$ is open in $X$, from which it easily
follows that $\hat s$ is continuous with this topology.

\begin{remark}
\label{rem:surjective}
{\rm Notice that if $\mathcal S$ is a surjective presheaf, then any element $a\in S_x$ can be written as
$a=\sup(\hat{s}_{n}(x))$, where $s_n$ is a rapidly increasing sequence in $S$. This is possible as the map $S\to S_x$ is surjective, hence
$a=\pi_x(s)$ for some $s\in S$, and $s=\sup s_n$ for such a sequence.}
\end{remark}

The following result gives another characterization of continuity that will
prove useful in the sequel.

\begin{proposition}\label{continuity}
Let $X$ be a compact Hausdorff space and $\mathcal{S}$ be a continuous surjective presheaf
on $\Cu$. Then, for a section $f\colon X\to F_S$, the following conditions are
equivalent:
\begin{enumerate}[{\rm (i)}]
\item $f$  is continuous.

\item For all $x\in X$ and $a_x\in S_x$ such that $a_x\ll f(x)$,
there exist a closed set $V$ with $x\in\mathring{V}$
and $s\in S$ such that $\hat{s}(x)\gg a_{x}$ and
$\hat{s}(y)\ll f(y)$ for all $y\in V$.
\end{enumerate}

\end{proposition}
\begin{proof}
 Let $f\colon X\to F_\mathcal{S}$ be a section satisfying \rm(ii) and consider
an open set of the form ${U}_{r}^{\gg}$ for some open set $U\subseteq X$ and $r\in S$. 
Then \[ \begin{array}{rl} f^{-1}(U_{r}^{\gg})& =\{y\in X\mid f(y)\gg \hat{r'}(y)\text{ for some } y\in U\text{ and for some }r'\gg r\} \\ & =\{y\in U\mid f(y)\gg \hat{r'}(y)\text{ for some }r'\gg r\}.\end{array}\]

For each $y$ in the above set there exists $r'\gg r$
such that $\hat{r'}(y)\ll f(y)$. Using property \rm(ii) there
exists $s\in S$ such that $\hat{r'}(y)\ll \hat{s}(y)
\ll f(y)$ and $\hat{s}(x)\ll f(x)$ for all $x\in \mathring{V}$ where $V$ is a closed set of $X$. 
Furthermore, we can find $r''\in S$ such that $r\ll r''\ll r'$, and use Lemma \ref{limit} to conclude that $\hat{r''}(z)\ll \hat s(z) \ll f(z)$ for all $z$ in an open set $W\subseteq X$,
proving that $f^{-1}(U_{r}^{\gg})$ is open. Therefore $f$ is continuous.

Now, let $f\colon X\to F_{\mathcal{S}}$ be continuous,
 $x\in X$ and $a_{x}\ll f(x)$. Using Remark \ref{rem:surjective}, we can write $f(x)=\sup(\hat{s}_{n}(x))$ where $(s_n)$ is a rapidly increasing sequence in $S$, and hence we can find $s\ll s'\in S$ such that
$$a_{x}\ll \hat{s}(x)\ll \hat{s'}(x)\ll f(x),$$ where $s,s'\in S$.

Let $U$ be any open neighborhood of $x$, and consider the open set $f^{-1}(U_{s}^{\gg})$. Note that it contains $x$ and that for any $z\in f^{-1}(U_s^{\gg})$, we have $f(z)\gg \hat{t}(z)$ for some $t\gg s$.
Hence, for any closed set $V$ contained in $f^{-1}(U_s^{\gg})$ such that $x\in \mathring V$, we have $f(z)\gg \hat{s}(z)$ for all $z\in V$. Thus, condition (ii) holds.
\end{proof}

Let $X$ be a compact Hausdorff space and let $\mathcal{S} $ be a continuous
presheaf on $\Cu$. We will denote the set of \emph{continuous sections of the space}
$F_S$ by $\Gamma(X,F_S)$, which is equipped with pointwise order and
addition. Notice that there is an order-embedding 
\[
\Gamma(X,F_S)\to \prod_{x\in
X} S_x
\]
(given by $f\mapsto (f(x))_{x\in X}$).

\begin{definition}
Let $X$ be a compact Hausdorff space. We say that a
$\mathrm{C}(X)$-algebra $A$ has no $K_1$ obstructions provided that, for all $x\in X$,
the fiber $A_x$ has stable rank one and $K_1(I)=0$ for any closed two-sided ideal of
$A_x$.
\end{definition}

The class just defined was already
considered, although not quite with this terminology, in \cite{APS}, where various aspects of the Cuntz semigroup of these
algebras were examined. We combine some of the ideas from \cite{APS} to prove the results below,
which are a first step towards the computation of the Cuntz semigroup of $\mathrm{C}(X)$-algebras without $K_1$ obstructions.

\begin{theorem}\label{inj.}
Let $X$ be a one dimensional compact Hausdorff space and $A$ be a $\mathrm{C}(X)$-algebra
without $K_1$ obstructions. Then, the map
$$\alpha\colon \Cu(A)\to \prod_{x\in X} \Cu(A_x),$$ given by
$\alpha[a]=([a(x)])_{x\in X}$ is an order embedding. In particular, $\alpha$ 
defines an order embedding
\[
\Cu(A)\to \Gamma(X,F_{\Cu(A)})\,.
\]
\end{theorem}

\begin{proof}
By our assumptions on $A$ and its fibers, we may assume that $A$ is stable.

Let $0<\epsilon<1$ be fixed, and let us suppose that $a,b\in A$ are positive contractions such that $a(x)\preceq b(x)$ for all $x\in X$.
Then, by the definition of the Cuntz order, since $A_x$ is a quotient of $A$
for each $x\in X$, there exists $d_x\in A$ such that
$$\|a(x)-d_x(x)b(x)d_x^*(x) \|<\epsilon.$$
By upper semicontinuity of the norm, the above inequality also holds in a neighborhood of $x$. Hence, since $X$ is a compact set, there exists a
finite cover of $X$, say $\{U_{i}\}^{n}_{i=1}$, and elements $(d_{i})^{n}_{i=1}\in A$ such that
$\|a(x)-d_{i}(x)b(x)d_{i}^*(x)\|<\epsilon$, for all $x\in \overline{U_{i}}$ and $1\leq i\leq n$. 
As $X$ is one dimensional, we may assume that $\{U_{i}\}$ and $\{\overline{U_i}\}$ have multiplicity at most two. 

Choose, by Urysohn's Lemma, functions $\lambda_i$ that are $1$ in the closed sets $U_i\setminus(\bigcup_{j\neq i}U_j)$ and $0$ in $U_i^c$.
Using these functions we define
$d(x)=\sum^{n}_{i=1}\lambda_{i}(x)d_{i}(x)$. Set $V=X\setminus (\bigcup_{i\neq j}(U_{i}\cap U_{j}))$ which is a closed set, and it 
is easy to check that $d$ satisfies \begin{equation}\label{V} \|a(x)-d(x)b(x)d^*(x)\|<\epsilon \end{equation}
for all $x\in V$.

Again,  choose for $i<j$ functions $\alpha_{i, j}$ such that $\alpha_{i,j}$ is one on $\overline{U_i\cap U_j}$ and zero on $\overline{U_k\cap U_l}$ whenever $\{k,l\}\neq\{i,j\}$. 
 We define $c(x)=\sum_{i<j}\alpha_{i,j}(x)d_{i}(x)$, put $U= (\bigcup_{i\neq j}(U_{i}\cap U_{j}))=V^c$ and notice that $c$ satisfies
\begin{equation}\label{U} \|a(x)-c(x)b(x)c^*(x)\|<\epsilon\end{equation} for all $x\in\overline{U}$. % Observe furthermore that $\overline{U}\cap V=\partial U=\partial V$ which is zero dimensional.

Now, by \cite[Lemma 2.2]{RK}, equations \eqref{V} and \eqref{U}, and taking into account that the norm of an element is computed fiberwise (\cite{Blan}), we have that 
\[
\pi_V((a-\epsilon)_+)\preceq \pi_V(b) \text{ and } \pi_{\overline U}((a-\epsilon)_+)\preceq \pi_{\overline U}(b)\,.
\]
Therefore
\[
([\pi_V(a-\epsilon)_+)],[\pi_{\overline U}(a-\epsilon)_+)])\leq ([\pi_V(b)],[\pi_{\overline U}(b)])
\]
in the pullback semigroup $\Cu(A(V))\oplus_{\Cu(A(\overline U\cap V)}\Cu(A(\overline U))$.
Since $A$ can also be written as the pullback $A=A(V)\oplus_{A(\overline{U}\cap V)}A(\overline{U})$ along the natural restriction maps (see \cite[Lemma 2.4]{Da}, and also \cite[Proposition 10.1.13]{Dixmier}), we can apply  \cite[Theorem 3.2]{APS}, to conclude that 
$(a-\epsilon)_+\preceq b$. Thus $a\preceq b$, and the result follows.
\end{proof}

\begin{corollary}\label{thm:sheaf}
Let $X$ be a one dimensional compact Hausdorff space, and let $A$ be a
$\mathrm{C}(X)$-algebra without $K_1$ obstructions. Then, $\Cu_A\colon \mathcal{V}_X\to\Cu$, $U\mapsto\Cu(A(U))$, is a surjective continuous sheaf.
\end{corollary}
\begin{proof}
We know already that $\Cu_A$ is a surjective continuous presheaf.
 Let $U$ and $V\in\mathcal{V}_X$ be such that $U\cap V\in\mathcal{V}_X$. 
Let $W=U\cup V$. 
We know then that $A(W)$ is isomorphic to the pullback $A(U)\oplus_{A(U\cap V)}A(V)$. 
Since $A(W)$ is a $\mathrm{C}(W)$-algebra without $K_1$ obstructions, 
we may apply Theorem \ref{inj.} to conclude that the map 
$\Cu(A(W))\to\prod_{x\in W}\Cu(A_x)$ (given by $[a]\mapsto ([a(x)])$) is an order-embedding.
 Then \cite[Theorem 3.3]{APS} implies that the natural map $\Cu(A(W))\to\Cu(A(U))\oplus_{\Cu(A(U\cap V))}\Cu(A(V))$ is surjective.
 Since it is also an order-embedding, by \cite[Theorem 3.2]{APS}, we obtain that it is an isomorphism.

\end{proof}

\section{Piecewise characteristic sections}

In this section we will show that, under additional assumptions, the map in
Theorem \ref{inj.} is also surjective, proving that there exists an isomorphism in
the category $\Cu$ between $\Cu(A)$ and $\Gamma(X,F_{\Cu(A)})$. 

Recall that, if $s\ll r\in S$, then $\pi_x(s)=\hat{s}(x)\ll
\hat{r}(x)=\pi_x(r)$ for all $x$. This comes from the fact that the induced
maps belong to the category $\Cu$, and so they preserve the compact
containment relation. We continue to assume throughout that $X$ is a compact Hausdorff space, which is also second countable. We shall use below $\partial (U)$ to denote the \emph{boundary} of a set $U$, that is, $\partial(U)=\overline{U}\setminus\mathring{U}$.

\begin{lemma}
\label{lem:suprema}
Let $\mathcal{S}\colon\mathcal{V}_X\to\Cu$ be a surjective presheaf of semigroups on $X$. 
\begin{enumerate}[{\rm (i)}]
\item
Let $f$, $g\in\Gamma(X,F_S)$, and $V$ a closed subset of $X$ such that
$f(y)\ll g(y)$ for all $y\in V$. Put
\[
g_{V,f}(x)=\left\{\begin{array}{cc} g(x) & \text{ if } x\notin V\\
                   f(x) & \text{ if } x\in V
                  \end{array}\right.
\]
Then $g_{V,f}\in\Gamma(X,F_S)$.
\item
If $g\in\Gamma(X,F_S)$ and $x\in X$, there exist a decreasing sequence $(V_n)$ of
closed sets (with $x\in\mathring{V}_n$ for all $n$) and a rapidly increasing
sequence $(s_n)$ in $S$ such that $g=\sup_n g_{V_n,s_n}$.
\end{enumerate}
\end{lemma}
\begin{proof}
(i): Using the fact that both $f$ and $g$ are continuous, it is enough to check that condition (ii) in Proposition \ref{continuity} is verified for $x\in\partial(V)$. 
Thus, let $a_x$ be such that $a_x\ll g_{V,f}(x)=f(x)\ll g(x)$. By continuity of $f$, there is a closed subset $U$ with $x\in \mathring{U}$ and $s\in S$ such that $a_x\ll\hat{s}(x)$
and $\hat{s}(y)\ll f(y)$ for all $y\in U$. As $s$ is a supremum of a rapidly increasing sequence, we may find $s'\ll s$ with $a_x\ll \hat{s'}(x)$.

Next, as $g$ is also continuous, there are $t\in S$ and a closed set $U'$ with $x\in\mathring{U'}$ such that $f(x)\ll \hat{t}(x)$ and $\hat{t}(y)\ll g(y)$ for all $y\in U'$.
Since $\hat s(x)\ll \hat t(x)$ and $s'\ll s$, we now use Lemma \ref{limit} to find $W$ with $\hat{s'}(y)\ll \hat t(y)$ for all $y\in W$. Now condition
 (ii) in Proposition \ref{continuity} is verified using the induced section $s'$ and the closed set $U\cap U'\cap W$.

(ii): Write $g(x)=\sup_n \hat{s}_n(x)$, where $(s_n)$ is a rapidly increasing
sequence in $S$ (see Remark \ref{rem:surjective}). 

Since $s_1\ll s_2$ and $g$ is continuous, condition (ii) of Proposition
\ref{continuity} applied to $\hat{s}_2(x)\ll g(x)$ yields $t\in S$ and a closed
set $U_1$ whose interior contains $x$ such that $\hat{s}_2(x)\ll\hat{t}(x)$ and
$\hat{t}(y)\ll g(y)$ for all $y\in V_1$. We now apply Lemma \ref{limit}, so
that there is another closed set $U_1'$ (with $x\in\mathring{U_1'}$) so that
$\hat{s}_1(y)\ll\hat{t}(y)$ for any $y\in U_1'$. Let $V_1=U_1\cap U_1'$ and for
each $y\in V_1$, we have $\hat{s}_1(y)\ll \hat{t}(y)\ll g(y)$. Continue in
this way with the rest of the $s_n$'s, and notice that we can choose the sequence $(V_n)$ in such a way that $\cap V_n=\{x\}$.
\end{proof}

Using the previous lemma we can describe compact containment in
$\Gamma(X,F_{\Cu(A)})$.

\begin{proposition}\label{ll}
Let $\mathcal{S}\colon\mathcal{V}_X\to \Cu$
be a surjective presheaf of semigroups on $X$. For $f,g$ in $\Gamma(X,F_{\Cu(A)})$, the
following statements are equivalent:
\begin{enumerate}[\rm(i)]
 \item $f\ll g$.
 \item For all $x\in X$ there exists $a_{x}$ with $f(x)\ll a_{x}\ll
g(x)$ and such that if $s\in S$ satisfies $a_{x}\ll \hat{s}(x)$ and
$\hat{s}(y)\ll g(y)$ for $y$ in a closed set $U$ whose interior contains $x$,
then there exists a closed set $V\subseteq U$ with $x\in \mathring{V}$
and $f(y)\leq \hat{s}(y)\leq g(y)$ for all $y\in V$.
\end{enumerate}
\end{proposition}
\begin{proof}
(i) $\implies$ (ii):
Given $x\in X$, use Lemma \ref{lem:suprema} to write $g=\sup_n g_{V_n,s_n}$,
where $(s_n)$ is rapidly increasing in $S$ and $(V_n)$ is a decreasing sequence
of closed sets whose interior contain $x$. Since $f\ll g$, there is $n$ such
that
$$f\leq g_{V_n,s_n}\leq g_{V_{n+1},s_{n+1}}\leq g\,.$$
Let $a_x=g_{V_{n+1},s_{n+1}}(x)=\hat{s}_{n+1}(x)$, which clearly satisfies
$f(x)\leq\hat{s}_n(x)\ll\hat{s}_{n+1}(x)\ll g(x)$. Assume now that $s\in S$ and
$U$ is a closed set with $x\in\mathring{U}$ such that $a_x\ll \hat{s}(x)$ and
$\hat{s}(y)\ll g(y)$ for all $y\in U$. Since $s_n\ll s_{n+1}$ and
$\hat{s}_{n+1}(x)\ll\hat{s}(x)$, there is by Lemma \ref{limit} a closed set $V$
with $x\in\mathring{V}$ (and we may assume $V\subset V_{n+1}\cap U$) such that
$\hat{s}_n(y)\leq \hat{s}(y)$ for all $y\in V$. Thus $f(y)\leq\hat{s}_n(y)\leq
\hat{s}(y)\leq g(y)$ for all $y\in V$.

(ii) $\implies$ (i):
Suppose now that $g\leq\sup(g_{n})$, where $(g_n)$ is an increasing sequence in $\Gamma(X, F_S)$.
Let $x\in X$, and write $g=\sup g_{V_n,s_n}$ as in Lemma \ref{lem:suprema}, where
$(s_n)$ is a rapidly increasing sequence in $S$. Our assumption provides us first with $a_x$
such that $f(x)\ll a_x\ll g(x)$. In particular, there is $m$ such that
$a_x\ll\hat{s}_m(x)\ll\hat{s}_{m+1}(x)\ll\hat{s}_{m+2}(x)\ll g(x)$, and hence
there exists $k$ with $\hat{s}_{m+1}(x)\ll g_k(x)$.

As $g_k$ is continuous, condition (ii) in Proposition \ref{continuity} implies
that we may find $s\in S$ and a closed set $U$ with $x\in\mathring{U}$ such
that $\hat{s}_{m+1}(x)\ll\hat{s}(x)$ and $\hat{s}(y)\ll g_k(y)$ for all $y\in
U$. Now, as $s_m\ll s_{m+1}$, there exists a closed subset $V$ with
$x\in\mathring{V}$ and $\hat{s}_m(y)\leq \hat{s}(y)$ for all $y\in V$, whence
$\hat{s}_m(y)\leq g(y)$ for all $y\in U\cap V$.

Since also $\hat{s}_m(y)\ll g(y)$ for all $y\in V_m$, there is by assumption a
closed set $W\subseteq V_m\cap V$ (whose interior contains $x$) such that
$f(y)\leq \hat{s}_m(y)\leq g_k(y)$ for all $y\in W$. Now, by a standard
compactness argument we may choose $l$ such that $f\leq g_l$.
\end{proof}

\begin{lemma}\label{waybelow} Let $\mathcal{S}\colon\mathcal{V}_X\to\Cu$ be a surjective presheaf of semigroups
Then, the morphism  $$\begin{array}{cccc}
                     \alpha:&   S&\to &\Gamma(X,F_S) \\
                            &  s&\mapsto &\hat{s}
                  \end{array}
$$  preserves compact containment and suprema.
\end{lemma}
\begin{proof}
Using condition (ii) of Proposition \ref{continuity} it is easy to verify
that, if $(f_n)$ is an increasing sequence in $\Gamma(X, F_S)$, then its
pointwise supremum is also a continuous section.

% Assume now that $s\ll r$ in $S$. Write
% $r=\sup(r_{n})$, where $(r_{n})$ is a rapidly increasing sequence in
% $S$. We may find $m$ such that
% $$s\ll r_m\ll r_{m+1}\ll r_{m+2}\ll r\,.$$ 
% 
% Take $a_x=\hat{r}_{m+2}(x)$. Suppose that $t\in S$ satisfies $a_x\ll \hat{t}(x)$
% and $\hat{t}(y)\ll\hat{r}(y)$ for $y$ in a closed subset $U$ whose interior
% contains $x$. By Lemma \ref{limit}, there is a closed set $V$ such that
% $x\in\mathring{V}$ and $\hat{r}_n(y)\leq \hat{t}(y)$ for $y\in V$. Thus, for
% any $y\in V\cap U$, we have $\hat{s}(y)\leq\hat{r}_n(y)\leq\hat{t}(y)\leq
% \hat{r}(y)$. This verifies condition (ii) in Proposition \ref{ll}, whence $\hat{s}\ll\hat{r}$.

Assume now that $s\ll r$ in $S$. Write
$r=\sup(r_{n})$, where $(r_{n})$ is a rapidly increasing sequence in
$S$. We may find $m$ such that
$$s\ll r_m\ll r\,.$$ 

Take $a_x=\hat{r}_{m}(x)$. Suppose that $t\in S$ satisfies $a_x\ll \hat{t}(x)$
and $\hat{t}(y)\ll\hat{r}(y)$ for $y$ in a closed subset $U$ whose interior
contains $x$. By Lemma \ref{limit}, applied to $\hat{r}_m(x)\leq \hat t(x)$ and $s\ll r_m$, there is a closed set $V$ such that
$x\in\mathring{V}$ and $\hat{s}(y)\leq \hat{t}(y)$ for $y\in V$. Thus, for
any $y\in V\cap U$, we have $\hat{s}(y)\leq\hat{t}(y)\leq \hat{r}(y)$. This verifies condition (ii) in Proposition \ref{ll}, whence $\hat{s}\ll\hat{r}$.

\end{proof}

\begin{corollary}
\label{cor:extensiontonbd}
Let $\mathcal{S}\colon \mathcal{V}_X\to\Cu$ be a surjective sheaf of semigroups on $X$, $f\in\Gamma(X,F_S)$, $s\in S$, and let $V$ be a closed subset of $X$. If $\hat{s}(x)\leq f(x)$ for all $x\in V$ and $s'\ll s$, then there is a closed subset $W$ of $X$ with $V\subset\mathring{W}$ such that $\pi_{W}(s')\ll f_{|W}$. 
\end{corollary}
\begin{proof}
Let $s'\ll t'\ll t\ll s$ in $S$. For each $x\in V$, there is by Proposition \ref{continuity} a closed set $U_x$ whose interior contains $x$, and $r_x\in S$ such that $\hat{t}(x)\ll\hat{r}_x(x)$, and $\hat{r}_x(y)\leq f(y)$ for all $y\in U_x$. Now apply condition (i) of Lemma \ref{limit} to $t'\ll t$ in order to find another closed set $V_x$ such that $x\in\mathring{V}_x$ and $\hat{t'}(y)\leq \hat{r}_x(y)$ for $y\in V_x$. Letting $W_x=U_x\cap V_x$, we have $\hat{t'}(y)\ll f(y)$ for all $y\in  W_x$. Since $V\subseteq \bigcup_{x} \mathring{W}_x$, and $V$ is closed, we may find a finite number of $W_x$'s that cover $V$, whose union is the closed set $W$ we are after. Since $\mathcal{S}$ is a sheaf, it follows that $\pi_{W}(t')\leq f_{|W}$, and by Lemma \ref{waybelow} we see that $\pi_{W}(s')\ll\pi_{W}(t')\leq f_{|W}$, as desired.
\end{proof}

We now proceed to define a class of continuous sections that will play an
important role. This will be a version, for presheaves on spaces of dimension one, 
of the notion of piecewise characteristic function given in \cite[Definitions 2.4 and 5.9]{APS}.
We show below that, for a surjective sheaf of semigroups $\mathcal{S}\colon \mathcal{V}_X\to \Cu$ on a one dimensional space $X$,
every element in $\Gamma(X,F_S)$ can be written as the supremum of a rapidly increasing sequence of piecewise characteristic sections.
From this, we can conclude that $\Gamma(X,F_S)$ is an object in $\Cu$. Just as in \cite{APS}, we could define piecewise characteristic sections 
for spaces of arbitrary (finite) dimension and make the case that $\Gamma(X,F_S)$ belongs to $\Cu$ as well.
This is however technically much more involved and beyond the scope of this paper, whence it will be pursued elsewhere.

\begin{definition}(Piecewise characteristic sections)\label{piecewise}
Let $X$ be a one dimensional compact Hausdorff space. Let
$\{U_i\}_{i=1\ldots n}$ be an open cover of $X$ such that the
multiplicity of $\{U_i\}$ and $\{\bar{U_i}\}$ is at most two. Assume also that
$\mathrm{dim}(\partial(\overline{U_i}))=0$ for all $i$.

Let $\mathcal{S}\colon\mathcal{V}_X\to\Cu$ be a presheaf of semigroups on $X$.
For $i\in\{1,\ldots,n\}$, choose elements $s_i\in S$ and $s_{\{i, j\}}\in S$ whenever $i\neq j$, such that
\[
\hat{s}_i(x)\leq \hat{s}_{\{i, j\}}(x) 
\text{ for all } x \text{ in }\overline{ \partial(U_i\cap U_j)\cap U_i}\,.
\]
We define a {\rm piecewise characteristic section} as 
\[
g(x)=\left\{ \begin{array}{cc}
                            \hat{s}_{i}(x)&\text{   if   }x\in U_i\setminus
(\bigcup_{j\neq i} U_j)\\
                            \hat{s}_{\{i, j\}}(x)&\text{   if   }x\in U_i\cap U_j\,\,\,\,\, .

                                                   \end{array} \right.
\]
\end{definition}

By an argument similar to the one in Lemma \ref{lem:suprema}, it follows that piecewise characteristic sections are continuous.

\begin{remark}
\label{rem:pwdim0}
{\rm In the case of zero dimensional spaces, piecewise characteristic sections are much easier to define. 
Given an open cover $\{U_i\}_{i=1,\ldots,n}$ consisting of pairwise disjoint clopen sets, a presheaf of 
semigroups $\mathcal{S}$ on $\Cu$ and elements $s_1,\ldots,s_n\in S$, a piecewise characteristic section
 in this setting is an element $g\in\Gamma(X,F_S)$ such that $g(x)=\hat{s}_i(x)$, whenever $x\in U_i$.}
\end{remark}

If $f\in\Gamma(X,F_S)$ and $g$ is a piecewise characteristic section such that  $g\ll f$, then we say that $g$ is a
piecewise characteristic section of $f$ and we will denote the set of these sections by $\chi(f)$.

\begin{lemma}\label{limit1}
Let $X$ be a one dimensional compact Hausdorff space, and
$\mathcal{S}\colon\mathcal{V}_X\to\Cu$ a surjective presheaf of semigroups.
If $f\in \Gamma(X,F_S)$, then $$f=\sup\{g\mid g\in\chi(f)\}\,.$$
\end{lemma}
\begin{proof}
Let $x\in X$. By Lemma \ref{lem:suprema}, we may write $f=\sup
f_{V_n,s_n}$, where $(V_n)$ is a decreasing sequence of closed sets with
$x\in\mathring{V_n}$ and $(s_n)$ is rapidly increasing. By construction,
$f_{V_n,s_n} (y)=\hat{s}_n(y)\ll f(y)$ for all $y\in V_n$.

Now define
\[
h_n(y)=\left\{ \begin{array}{cc}
                            \hat{s}_n(y) &\text{   if  
}y\in\mathring{V_n}\\
                            0& \text{ otherwise}\,\,\,\,\, .

                                                   \end{array} \right.
\] 
It is easy to verify, using Proposition \ref{ll}, that $h_n\ll f$, and also that each $h_n$ is a piecewise characteristic section for $f$. Using this fact for all $x\in X$, we conclude that $f=\sup \{ g \mid g \in \chi(f)\}$.
\end{proof}

\begin{proposition}\label{middle}
Let $X$ be a compact Hausdorff space with $\mathrm{dim}(X)\leq 1$, and let
$\mathcal{S}\colon\mathcal{V}_X\to\Cu$ be a surjective sheaf of semigroups. Suppose $h_1,h_2,f\in\Gamma(X,F_S)$ such that $h_1,h_2\ll f$. Then, there exists $g\in\chi(f)$ such that
$h_1,h_2\ll g$. In particular, $\chi(f)$ is an upwards directed set.
\end{proposition}
\begin{proof}
Assume first that $X$ has dimension 0. Writing $f$ as in condition (ii) of Lemma~\ref{lem:suprema} we can find, for each 
$x\in X$, an open set $V_x$ that contains $x$, and elements $s'_x\ll s_x \ll s''_x\in S$ such that 
 \begin{equation}\label{dim0} h_1(y),h_2(y)\ll \hat {s'_x}(y) \ll \hat s_x(y)\ll \hat {s''_x}(y)\ll f(y)\quad \text{ for all }y\in \overline{V_x}.\end{equation}
Using compactness and the fact that $X$ is zero dimensional, there are $x_1,\ldots,x_n\in X$ and (pairwise disjoint) clopen sets $\{V_i\}_{i=1,\ldots,n}$ with $V_i\subseteq V_{x_i}$ and such that $X=\cup_i V_i$. Put $s_i=s_{x_i}$, $s'_i=s'_{x_i}$ and $s''_i=s''_{x_i}$. Define, using this cover, a piecewise characteristic section $g$ as $g(x)=\hat{s}_i(x)$ if $x\in V_i$. 
It now follows from \eqref{dim0} that $h_1,h_2\ll g\ll f$ (the elements $s'_i,s''_i$ are used here to obtain compact containment).

We turn now to the case where $X$ has dimension 1, and start as in the previous paragraph, with some additional care. Choose, for each $x$, a $\delta_x$-ball $V''_x$  (where $\delta_x>0$) centered at $x$ and elements $s'_x\ll s_x\ll s''_x$ such that condition \eqref{dim0} is satisfied (for all $y\in \overline{V''_x}$). Denote by $V'_x\subseteq V''_x$ the cover consisting of $\delta_x/2$-balls. By compactness we obtain a finite cover $\{V'_{x_1},\dots,V'_{x_n}\}$. 
Using \cite[Lemma 8.1.1]{Pears} together with the fact that $X$ has dimension 1, this cover has a refinement $\{V_i\}_{i=1}^n$ such that $\{V_i\}$ and $\{\overline{V}_i\}$ have both multiplicity at most $2$ and such that 
$\partial(V_i)$ has dimension $0$ for each $i$. As before, set $s_i=s_{x_i}$, $s_i'=s'_{x_i}$ and $s_i''=s''_{x_i}$. 

Let $Y$ be the closed set $\cup_{i}\partial(V_i)$, which also has dimension $0$. Put $\delta=\mathrm{min}\{\delta_{x_i}/3\}$. By construction, there is a $\delta$-neighborhood $V_i^\delta$ such that $V_i^\delta\subseteq V_i''$. As in the proof of Lemma \ref{limit1}, we see that the sections
\[
g_i(y)=\left\{ \begin{array}{cc}
                            \hat{s''_i}(y) &\text{   if  
}y\in V_i^{\delta}\\
                            0& \text{ otherwise}

                                                   \end{array} \right.
\] 
satisfy $g_i\ll f$. 
We now restrict to $Y$ and proceed as in the argument of the zero dimensional case above. In this way, we obtain piecewise characteristic sections $g_Y$, $g'_Y$, $g''_Y\in \Gamma(Y,F_S)$, defined by some open cover $\{W_i\}_{i=1}^m$ (of pairwise disjoint clopen sets of $Y$) and elements $t_i\ll t'_i\ll t''_i\in S$ in such a way that $g_Y(y)=\hat{t}_i(y)$, $g'_Y(y)=\hat{t'_i}(y)$ and $g''_Y(y)=\hat{t''_i}(y)$ whenever $y\in W_i$, and such that 
\begin{equation}\label{induccio0} \pi_Y(g_i)\ll g_Y\ll g'_Y \ll g''_Y\ll\pi_Y(f) \text{ for all }i=1,\dots, n\,.\end{equation}

Observe that we can choose the $W_i$ of arbitrarily small size, thus in particular we may assume that each one is contained in a $\delta/6$-ball. In this way, whenever $\overline{W}_i\cap \overline{V}_j\neq \emptyset$, we have $W_i\subseteq V_j^\delta$. Therefore, if $x\in W_i$, it follows from \eqref{induccio0} that  
\[ 
\hat {s''_j}(x)=g_j(x)\leq g_Y(x)=\hat t_i(x)\,.
\]

By condition (ii) in Lemma \ref{limit}, applied to the previous inequality, there is $\epsilon>0$ such that $\hat s_j(x)\leq \hat t_i(x)$ for all $x\in W_i^\epsilon$.
Since the $W_i$ are pairwise disjoint clopen sets, we can choose $\epsilon$ such that the sets $W_i^\epsilon$ are still pairwise disjoint. Further, since also $\hat{t''_i} (y)\leq f(y)$ for $y\in W_i$ and $t'_i\ll t''_i$, we may apply Corollary \ref{cor:extensiontonbd} to obtain $\pi_{\overline{W_i^{\epsilon}}}(t'_i)\ll \pi_{\overline{W_i^{\epsilon}}}(f)$ (further decreasing $\epsilon$ if necessary). As for each $i$, we can find $U_i$ with $\overline{W_i^{\epsilon/2}}\subseteq U_i\subseteq W_i^{\epsilon}$ with zero dimensional boundary, after a slight abuse of notation we shall assume that $W_i^{\epsilon}$ itself has zero dimensional boundary. Put $Y^\epsilon=\cup_{i=1}^m W_i^\epsilon$. Notice now that, for $i$, $k<l$, the closed sets $V_i\setminus (Y^\epsilon\cup \cup_{j\neq i}V_j)$ and $(V_k\cap V_l)\setminus Y^{\epsilon}$ are also pairwise disjoint, whence they admit pairwise disjoint $\epsilon'$-neighborhoods (for a sufficiently small $\epsilon'$). As before, we shall also assume these neighborhoods have zero dimensional boundary.

Now consider the cover that consists of the sets 
\[
\{W_i^{\epsilon}\,,i=1,\ldots, m\,,(V_i\setminus (Y^\epsilon\cup \cup_{j\neq i}V_j))^{\epsilon'}\,,i=1,\ldots, n\,,(V_k\cap V_l\setminus (Y^{\epsilon}))^{\epsilon'}\,,k<l\}\,,
\]
and define a piecewise characteristic section $g$ as follows
\[
g(x)=\left\{ \begin{array}{cl}
                            \hat{t}_i(x) &\text{   if  
}x\in W_i^{\epsilon}\\
                            \hat{s}_i(x) & \text{ if }x\in (V_i\setminus (Y^\epsilon\cup \cup_{j\neq i}V_j))^{\epsilon'}\setminus Y^\epsilon\\
                            \hat{s}_k(x) & \text{ if }x\in (V_k\cap V_l\setminus (Y^{\epsilon}))^{\epsilon'}\setminus Y^{\epsilon} \text{ for }k<l\,.

                                                   \end{array} \right.
\]
That $h_1$, $h_2\leq g$ follows by construction of $g$. It remains to show that $g\ll f$. This also follows from our construction, using condition (ii) of Proposition \ref{ll}. For example, given $x\in W_i^{\epsilon}$, we have
\[
g(x)= \hat{t_i}(x)\ll\hat{t'_i}(x)\ll f(x)\,.
\]
If now $s\in S$ satisfies that $\hat{t'_i}(x)\ll\hat{s}(x)$ and $\hat{s}(y)\ll f(y)$ for $y$ in a closed set (whose interior contains $x$) then, since $t_i\ll t'_i$, we may find (again by Lemma \ref{limit}) a smaller closed set (contained in $W_i^{\epsilon}$ and with interior containing $x$) such that $g(y)=\hat{t_i}(y)\leq \hat{s}(y)\leq f(y)$ for $y$ in that set.
\end{proof}

\begin{proposition}\label{increasing}
Let $X$ be a one dimensional compact Hausdorff space, and let $\mathcal{S}\colon\mathcal{V}_X\to\Cu$ be a surjective sheaf of
semigroups with $S$ countably based. If $f\in\Gamma(X,F_S)$, then $f$ is the supremum of a rapidly
increasing sequence of elements from $\chi(f)$.
\end{proposition}
\begin{proof}
Let us define a new topology on $F_S$. Let
$s\in S$ and let $U$ be an open set in $X$. Consider the topology generated by the sets 
\[
U_{s}^{\ll}=\{y\in F_S\mid \hat{s}(x)\gg y\text{ for some }
x\in U\}\,.
\]
We claim that, under this topology, $F_S$ is second countable. 
Let $\{U_{n}\}$ be a basis of $X$, and $\{s_{n}\}_{n\in\mathbb{N}}$ be a dense subset of $S$. 
Therefore the collections of sets $\{(U_{n})^{\ll}_{s_{i}}\}_{n,i\in\mathbb{N}}$ is a countable basis for $F_S$. Indeed, given an open set $U$ of $X$ and $s\in S$, find  
sequences $(U_{n_i})$ and $(s_{m_j})$ such that $U=\cup U_{n_i}$ and $s=\sup s_{m_j}$. Then
\[
U_{s}^{\ll}=\cup
(U_{n_i})_{s_{m_j}}^{\ll}\,.
\]
Now, for $f\in\Gamma(X,F_S)$, put $U_{f}=\{a_{x}\in F_S \mid a_{x}\ll
f(x)\text{ for }x\in U\} $. This set is open in the topology we just have defined. To see this, let $a_x\in U_f$, and invoke Proposition \ref{continuity} to find an open set $V$ and $s\in S$ such that $a_x\ll\hat{s}(x)$ and $\hat{s}(y)\ll f(y)$ for all $y\in V$. It then follows that $a_x\in V_s^{\ll}\subseteq U_f$.

Using Lemma \ref{limit1}, we see that $U_{f}=\cup_{g\in\chi(f)}U_{g}$. Since $F_S$ is second countable, it has the Lindel\"of property, whence we may find a sequence $(g_n)$ in $\chi(f)$ such that $U_{f}=\cup_{n} U_{g_{n}}$. This sequence may be taken to be increasing by Proposition \ref{middle}. Translating this back to $\Gamma(X,F_S)$, we get $f=\sup (g_{n})$.
\end{proof}

Assembling our observations we obtain the following:
\begin{theorem}
Let $X$ be a one dimensional compact Hausdorff space, and let $\mathcal{S}\colon\mathcal{V}_X\to\Cu$ be a surjective sheaf of semigroups such that $S$
is countably based. Then, the semigroup $\Gamma(X,F_S)$ of continuous sections
belongs to the category $\Cu$.
\end{theorem}

The next result shows, in a particular case, the existence of an induced section between any two compactly contained
piecewise characteristic sections. 

\begin{proposition}\label{induced}
Let $X$ be a one dimensional compact Hausdorff space and let $A$ be a stable continuous field over $X$ with no $K_1$ obstructions. 
Let $f\ll g$ be elements in $\Gamma(X,F_{\Cu(A)})$ such that $g$ is a piecewise characteristic section. Then there exists an element $h\in A$ which satisfies $f(x)\leq [\pi_x (h)] \leq g(x)$ for all $x\in X$.
\end{proposition}
\begin{proof}
Since $g$ is a piecewise characteristic section there is a cover $\{U_i\}_{i=1}^n$ of $X$ such that both $\{U_i\}$ and $\{\overline{U}_i\}$ have multiplicity at most $2$, and there are elements $[a_i]$, $[a_{\{i,j\}}]$ in $\Cu(A)$ which are the values that $g$ takes (according to Definition \ref{piecewise}).

For $\epsilon >0$, let $g_{\epsilon}$ be the section defined on the same cover as $g$ and that takes values $[(a_i-\epsilon)_+]$, $[a_{\{i,j\}}]$. As $g=\sup_{\epsilon} g_{\epsilon}$ and $f\ll g$, we may choose $\epsilon>0$ such that $f\leq g_{\epsilon}$, and in particular
\[
f(x)\leq \pi_x([(a_i-\epsilon)_+])\ll \pi_x([ a_i]) \text{ for all } x \text{ in }U_i\setminus (\cup_{j\neq i} U_j)\,.
\]
Notice now that the closed sets $\overline{\partial(U_i\cap U_j)\cap U_i}$ and $\overline{\partial(U_k\cap U_l)\cap U_l}$ are pairwise disjoint whenever $(i,j)\neq (k,l)$. (This follows from elementary arguments together with the assumption that the cover $\{\overline{U}_i\}$ has multiplicity at most $2$.)

Furthermore, by definition of $g$ we have $ \pi_x([a_i])\leq \pi_x([ a_{\{i,j\}}])$ for all $x\in \overline{\partial(U_i\cap U_j)\cap U_i}$. Therefore, there exists by Corollary \ref{cor:extensiontonbd} a neighborhood $W_{i,j}$ of $\overline{\partial(U_i\cap U_j)\cap U_i}$ for which 
\[
\pi_{\overline{W}_{i,j}}([a_i])\leq\pi_{\overline{W}_{i,j}}([a_{\{i,j\}}])\,.
\] 
We may assume without loss of generality that 
%each the $W_{i,j}$
%is a $\delta$-neighborhood of $\overline{\partial(U_i\cap U_j)\cap U_i}$ (for a fixed $\delta>0$), and that 
the closures $\overline{W}_{i,j}$ are  pairwise disjoint sets. Since also $\overline{\partial(U_i\cap U_j)\cap U_i}\cap\overline{U}_k=\emptyset$ whenever $k\neq i,j$, we may furthermore assume that $W_{i,j}\cap \overline{U}_k=\emptyset$ for $k\neq i,j$.

By Proposition \ref{here} there exist unitaries $u_{i,j}\in \mathcal{U}(A(\overline{W}_{i,j})^{\sim})$ such that
\[
u_{i,j}\pi_{\overline{W}_{i,j}}((a_i-\epsilon)_{+})u_{i,j}^{*}\in \mathrm{Her}(\pi_{\overline{W}_{i,j}}(a_{\{i,j\}}))\,.
\] 

Now, as $A$ and $A(\overline{W}_{i,j})$ are stable, the unitary groups of their multiplier algebras are connected in the norm topology (see, e.g. \cite[Corollary 16.7]{wo}). Furthermore, since the natural map $\pi_{\overline{W}_{i,j}}\colon A\to A(\overline{W}_{i,j})$ induces a surjective morphism $\mathcal{M}(A)\to\mathcal{M}(A(\overline{W}_{i,j}))$ (by, e.g. \cite[Theorem 2.3.9]{wo}), we can find, for each unitary $u_{i,j}$, a unitary lift $\tilde{u}_{i,j}$ in $\mathcal{M}(A)$.

We now have continuous paths of unitaries $w_{i,j}\colon  [0,1] \to\mathcal{U}(\mathcal{M}(A))$ such that $w_{i,j}(0)=1$ and $w_{i,j}(1)=\tilde{u}_{i,j}$. Put $\gamma=\mathrm{min}\{\mathrm{dist}(\overline{W}_{i,j},\overline{W}_{k,l}\mid (i,j)\neq (k,l)\}$. Note that $\gamma>0$ as the sets $\overline{W}_{i,j}$ are pairwise disjoint. For $x\in X$, define a unitary in $\mathcal{M}(A)$ by 
\[
w_{i,j}^x=
	w_{i,j}\left(\frac{(\gamma-\mathrm{dist}(x,W_{i,j}))_+}{\gamma}\right)\,.
\]
Observe that, if $x\in W_{k,l}$, then $w_{i,j}^x=\tilde{u}_{i,j}$ if $(k,l)=(i,j)$ and equals $1$ otherwise. Now put
\[
w_i^x=\prod_{j}w_{i,j}^x\,.
\]
Since each $\pi_x$ is norm decreasing and the $w_i^x$ are defined by products and compositions of continuous functions,
we obtain, using \cite[Definition 10.3.1]{Dixmier}, that for each $c\in A$, the tuple $(\pi_x(w_i^xc))_{x\in X}\in \prod_{x\in X} A_x$ defines fiberwise an element 
in $A$ which we denote by $w_ic$. 

Now let $\{\lambda_i\}_i$ be continuous positive real-valued functions
on $[0,1]$ whose respective supports are $\{(U_i\setminus (\cup_{j\neq i} U_j))\cup (\cup_j W_{i,j}) \}_{i}$ and 
$\{\lambda_{\{i,j\}}\}_{i,j}$ with supports $\{U_i\cap U_j\}_{i,j}$. Define the following element in $A$
\[h=\sum_{i} \lambda_i w_{i}(a_i-\epsilon)_+ w_{i}^{*} +\sum_{i\neq j} \lambda_{\{i,j\}} a_{\{i,j\}}.\]

We now check that $[\pi_x (h)]=g_\epsilon(x)$, and this will yield the desired conclusion. 

%We denote $W=\cup_{i,j} W_{i,j}$. 
%On the one hand,
If $x\in U_i\setminus (\cup_{j\neq i} U_j)$, then 
$\pi_x(h)=\lambda_i(x)\pi_x(w_i(a_i-\epsilon)_+w_i^{*})$ where $\lambda_i(x)\neq 0$, and this is equivalent to $\pi_x((a_i-\epsilon)_+)$. Hence $[\pi_x(h)]=g_\epsilon(x)$. 

On the other hand, if $x\in U_i\cap U_j$ for some $i,j$ then $\lambda_{\{i,j\}}(x)\neq 0$, and
\[ \pi_{x}(h)=\left\{ \begin{array}{ll}
\lambda_{i}(x)\pi_x(\tilde{u}_{i,j}(a_i-\epsilon)_+\tilde{u}_{i,j}^*) + \lambda_{\{i,j\}}(x)\pi_{x}(a_{\{i,j\}}) & \text{ if }  x\in U_i\cap U_j \cap W_{i,j}\\
\lambda_{j}(x)\pi_x(\tilde{u}_{j,i}(a_j-\epsilon)_+\tilde{u}_{j,i}^*) + \lambda_{\{i,j\}}(x)\pi_{x}(a_{\{i,j\}}) & \text{ if }  x\in U_i\cap U_j \cap W_{j,i}\\
    \lambda_{\{i,j\}}(x)\pi_{x}(a_{\{i,j\}})  & \text{ if } x\in U_i\cap U_j\setminus (W_{i,j}\cup W_{j,i})\,.
\end{array}\right. \]
If, for example, $x\in U_i\cap U_j \cap W_{i,j}$, then $\pi_x(\tilde{u}_{i,j}(a_i-\epsilon)_+\tilde{u}_{i,j}^*)\in \mathrm{Her}(\pi_x (a_{\{i,j\}}))$, and we conclude that $[\pi_{x}(h)]=[\pi_x(a_{\{i,j\}})]=g_{\epsilon}(x)$. The other cases are treated similarly. 

\end{proof}

This last result (together with Proposition \ref{increasing}) proves that, with some restrictions on $X$ and $A$, the set of induced sections is a dense subset of $\Gamma(X,F_{\Cu(A)})$, that is, every element in $\Gamma(X,F_{\Cu(A)})$ is a supremum of a rapidly increasing sequence of induced sections.

\begin{theorem}\label{isomorfisme}
Let $X$ be a one dimensional compact Hausdorff space and let $A$ be a continuous field over $X$ without $K_1$ obstructions. Then, the map 
\[
\begin{array}{cccc}
                     \alpha:&   \Cu(A)&\to &\Gamma(X,F_{\Cu(A)}) \\
                            & s &\mapsto &\hat{s}
                  \end{array}
\]
is an order isomorphism in $\Cu$.
\end{theorem}
\begin{proof}
Let $f$ be a continuous section in $\Gamma(X,F_{\Cu(A)}) $ and use Propositions \ref{increasing} and \ref{induced} to 
write $f$ as the supremum of a rapidly increasing sequence of induced sections $f=\sup_{n} \hat s_n$. Since $\alpha$ is an order embedding (by Theorem \ref{inj.}) and $\alpha(s_n)= \hat s_n$, the 
sequence $s_n$ is also increasing in $\Cu(A)$ and thus we can define $s=\sup_n s_n\in \Cu(A)$. The result now follows using Lemma \ref{waybelow}.
\end{proof}

Since the conditions on the fibers in the previous Theorem are satisfied by simple AI-algebras we obtain the following.

\begin{corollary}
 Let  $X$ be a one dimensional compact Hausdorff space and let $A$ be a continuous field over $X$ such that $A_x$ is a simple AI-algebra for all $x\in X$. Then $\Cu(A)\cong \Gamma (X,F_{\Cu(A)})$.
\end{corollary}

\section{The sheaf $\Cu_A(\_)$}

For a compact Hausdorff space $X$, denote by $\mathcal{C}_X$ the category whose 
objects are the $\mathrm{C}(X)$-algebras, and the morphisms between objects are those $^*$-homomorphisms such that commute with the (respective) structure maps.

Denote by $\mathcal{S}_{\Cu}$ the category which as objects has the presheaves $\Cu_A(\_)$ on $X$, where 
$A$ belongs to $\mathcal{C}_X$, and the maps are presheaf homomorphisms.
The following holds by definition:
\begin{lemma}
The assignment
\[
\begin{array}{cccc}
    \Cu_{(\_)}\colon & \mathcal{C}_X&\to& \mathcal{S}_{\Cu}\\
    &A&\mapsto& \Cu_A(\_)\,
    \end{array}
\]
is a covariant functor.
\end{lemma}

\begin{theorem}
\label{prop:isoseccions}
Let $X$ be a one dimensional compact Hausdorff space and let $A$ be a continuous field over $X$ without $K_1$ obstructions. Consider the functors
\[
\begin{array}{ccccccccccc}
    \Cu_{A}(\_):& \mathcal{V}_X&\to    & \Cu       &    &   \text { and }   & &     \Gamma(\_,F_{\Cu_{A(\_)}})\colon & \mathcal{V}_X&\to    & \Cu        \\
                        & V          &\mapsto& \Cu(A(V)) &    &                   & &                               & V          &\mapsto& \Gamma(V,F_{\Cu_{A(V)}})\,.
    \end{array}
\]
Then,  $\Cu_{A}(\_)$ and $\Gamma(\_,F_{\Cu_{A(\_)}})$ are isomorphic sheaves.
\end{theorem}
\begin{proof}
That  $\Cu_{A}(\_)$ is a sheaf follows from Corollary \ref{thm:sheaf}. Let $(h_V)_{V\in\mathcal{V}_X}$ be the collection of isomorphisms $h_V\colon\Cu(A(V))\to \Gamma(V,F_{\Cu_{A(V)}})$ described in Theorem \ref{isomorfisme}. 
Since, whenever $V\subset U$, the following diagram
\[
\xymatrix{
\Cu(A(V))\ar[r] & \Gamma(V,F_{\Cu_{A(V)}})\\
\Cu(A(U))\ar[u] _{(\Cu_{A}(\_))^{U}_{V}}\ar[r]&\Gamma(U,F_{\Cu_{A(U)}})\ar[u]_{(\Gamma(\_,F_{\Cu_{A(\_)}}))^{U}_{V}}
}
\]
clearly commutes, $(h_V)_{V\in\mathcal{V}_X}$ defines an isomorphism of sheafs $h\colon \Cu_{A}(\_)\to \Gamma(\_,F_{\Cu_{A(\_)}})$.
\end{proof}

In order to relate the Cuntz semigroup $\Cu(A)$ and the sheaf $\Cu_A(\_)$,  we now show that there exists an action of $\Cu(\mathrm{C}(X))$ on $\Cu(A)$ when $A$ is a $\mathrm{C}(X)$-algebra, which is naturally induced from the $\mathrm{C}(X)$-module structure on $A$.
\begin{definition}
\label{def:bimorphism}
Let $S$, $T$, $R$ be semigroups in $\Cu$. A $\ll$-bimorphism is a map $\varphi\colon S\times T\to R$ such that the map $\varphi(s,\_)\colon T\to R$, $s\in S$ (respectively, $\varphi(\_,t)\colon S\to R$, $t\in T$), preserves order, addition, suprema of increasing sequences, and moreover $\varphi(s',t')\ll\varphi(s,t)$ whenever $s'\ll s$ in $S$ and $t'\ll t$ in $T$.
\end{definition}

\begin{remark}
\label{rem:commute} {\rm We remark that if $A$ is a C$^*$-algebra and $a$, $b$ are commuting elements, then for $\epsilon >0$ we have $(a-\epsilon)_+(b-\epsilon)_+\preceq (ab-\epsilon^2)_+$. Indeed, since the C$^*$-subalgebra generated by $a$ and $b$ is commutative, Cuntz comparison is given by the support of the given elements, viewed as continuous functions on the spectrum of the algebra. It is then a simple matter to check that $\mathrm{supp}((a-\epsilon)_+(b-\epsilon)_+)\subseteq \mathrm{supp}((ab-\epsilon^2)_+)$.}
\end{remark}

\begin{proposition}
\label{prop:bimorphism}
Let $A$ and $B$ be stable and nuclear C$^*$-algebras. Then, the natural bilinear map $A\times B\to A\otimes B$ given by $(a,b)\mapsto a\otimes b$ induces a $\ll$-bimorphism
\[
\begin{array}{ccc} \Cu(A)\times\Cu(B) & \to  &\Cu(A\otimes B)\\
([a],[b]) & \mapsto & [a\otimes b]
\end{array}
\]
\end{proposition}
\begin{proof}
Since $A$ is stable, we may think of $\Cu(A)$ as equivalence classes of positive elements from $A$. We also have an isomorphism $\Theta\colon M_2(A)\to A$ given by isometries $w_1$, $w_2$ in $\mathcal{M}(A)$ with orthogonal ranges, so that $\Theta(a_{ij})=\sum_{i,j}w_ia_{ij}w_j^*$. Thus, in the Cuntz semigroup, $[a]+[b]=[\Theta\left(\begin{smallmatrix} a & 0 \\ 0 & b \end{smallmatrix}\right)]$.

The map $\Cu(A)\times\Cu(B)\to \Cu(A\otimes B)$ given by $([a],[b])\mapsto [a\otimes b]$ is well defined and order-preserving in each argument, by virtue of \cite[Lemma 4.2]{R}. Let $a$, $a'\in A_+$, $b\in B_+$. As
\[
[(w_1aw_1^*+w_2a'w_2^*)\otimes b]=[w_1aw_1^*\otimes b]+[w_2a'w_2^*\otimes b]=[a\otimes b]+[a'\otimes b]\,,
\]
we see that it is additive in the first entry (and analogously in the second entry).

Next, observe that if  $\|a\|,\|b\|\leq 1$, $\epsilon>0$,
\[
\|a\otimes b-(a-\epsilon)_{+}\otimes(b-\epsilon)_{+}\|\leq\|a\otimes b-(a-\epsilon)_{+}\otimes b\|+\|(a-\epsilon)_{+}\otimes b-(a-\epsilon)_{+}\otimes(b-\epsilon)_{+}\|\leq
\]
\[
\epsilon\|b\|+\|(a-\epsilon)_{+}\|\epsilon\leq 2\epsilon\,,
\]
and this implies $[ a\otimes b]=\sup([ (a-\epsilon)_{+}\otimes(b-\epsilon)_{+}])$. If now $[a]=\sup_n [a_n]$ for an increasing sequence $([a_n])$, then for any $[b]$ we have $[a_n\otimes b]\leq [a\otimes b]$. Given $\epsilon>0$, find $n$ with $[(a-\epsilon)_+]\leq [a_n]$, hence $[(a-\epsilon)_+\otimes (b-\epsilon)_+]\leq [a_n\otimes b]\leq\sup [a_n\otimes b]$. Taking supremum when $\epsilon$ goes to zero we obtain $[a\otimes b]=\sup [a_n\otimes b]$.

Finally, assume that $[a']\ll [a]$ in $\Cu(A)$, and $[b']\ll [b]$ in $\Cu (B)$. Find $\epsilon>0$ such that $[a']\leq [(a-\epsilon)_+]$ and $[b']\leq [(b-\epsilon)_+]$. Then $[a'\otimes b']\leq [(a-\epsilon)_+\otimes (b-\epsilon)_+]$.

Note that $(a-\epsilon)_+\otimes (b-\epsilon)_+\in A\otimes B\subseteq \mathcal{M}(A)\otimes\mathcal{M}(B)$ and, viewed in the tensor product of the multiplier algebras, we have $(a-\epsilon)_+\otimes (b-\epsilon)_+=((a-\epsilon)_+\otimes 1)(1\otimes(b-\epsilon)_+)$. Since $\mathcal{M}(A)\to \mathcal{M}(A)\otimes\mathcal{M}(B)$, $c\mapsto c\otimes 1$ is a $^*$-homomorphism, it induces a semigroup homomorphism $\Cu(\mathcal{M}(A))\to \Cu(\mathcal{M}(A)\otimes\mathcal{M}(B))$ in the category $\Cu$ and, in particular, since $[(a-\epsilon)_+]\ll [a]$ in $\Cu(\mathcal{M}(A))$, it follows that $[(a-\epsilon)_+\otimes 1]\ll [a\otimes 1]$ in $\Cu(\mathcal{M}(A)\otimes \mathcal{M}(B))$. Likewise, $[1\otimes (b-\epsilon)_+]\ll [1\otimes b]$, hence we may find $\epsilon'>0$ such that $[(a-\epsilon)_+\otimes 1]\leq [(a\otimes 1-\epsilon')_+]$ and $[1\otimes (b-\epsilon)_+]\leq [(1\otimes b-\epsilon')_+]$. Since the elements $(a-\epsilon)_+\otimes 1$, $(a\otimes 1-\epsilon')_+$, $1\otimes (b-\epsilon)_+$ and $(1\otimes b-\epsilon')_+$ all commute (and using Remark \ref{rem:commute}), it follows that
\begin{align*}
[(a-\epsilon)_+\otimes (b-\epsilon)_+] & =  [((a-\epsilon)_+\otimes 1)(1\otimes (b-\epsilon)_+)]\\ & \leq [(a\otimes 1-\epsilon')_+(1\otimes b-\epsilon')_+]\\ & \leq [(a\otimes b-\epsilon'^2)_+]\ll [a\otimes b]\,,
\end{align*}
whence $[a'\otimes b']\ll [a\otimes b]$.
\end{proof}
\begin{corollary}
\label{cor:action} Let $X$ be a compact Hausdorff space, and let $A$ be a stable $\mathrm{C}(X)$-algebra (with structure map $\theta$). Then the natural map $\mathrm{C}(X)\times A\to A$, given by $(f,a)\to \theta(f)a$ induces a $\ll$-bimorphism
\[
\gamma_A\colon \Cu(\mathrm{C}(X))\times\Cu(A)\to \Cu(A)
\]
such that maps $([f],[a])$ to $[\theta(f)a]$, for $f\in \mathrm{C}(X)_+$ and $a\in A_+$.
\end{corollary}
\begin{proof}
Since $\Cu(\mathrm{C}(X))=\Cu(\mathrm{C}(X)\otimes\mathcal K)$, Proposition \ref{prop:bimorphism} tells us that the map
\[
\begin{array}{ccc} \Cu(\mathrm{C}(X)\otimes\mathcal K)\times \Cu(A) &\to &\Cu(\mathrm{C}(X)\otimes\mathcal{K}\otimes A)\\
([f],[a]) & \mapsto & [f\otimes a]\end{array}
\]
is a $\ll$-bimorphism. Now the result follows after composing this map with the isomorphism $\Cu(\mathrm{C}(X)\otimes\mathcal{K}\otimes A)\cong \Cu(\mathrm{C}(X)\otimes A)$, followed by the map $\Cu(\theta)\colon \Cu(\mathrm{C}(X)\otimes A)\to \Cu(A)$.
\end{proof}

In what follows, we shall refer to the $\ll$-bimorphism $\gamma_A$ above as the action of $\Cu(\mathrm{C}(X))$ on $\Cu(A)$. If $A$ and $B$ are $\mathrm{C}(X)$-algebras, we will say then that a morphism $\varphi\colon \Cu(A)\to \Cu(B)$ \emph{preserves the action} provided $\varphi(\gamma_A(x,y))=\gamma_B(x,\varphi(y))$. Notice that this is always the case if $\varphi$ is induced by a $^*$-homomorphism of $\mathrm{C}(X)$-algebras. We will write $\gamma$ instead of $\gamma_A$, and we will moreover use the notation  $xy$ for $\gamma(x,y)$.

As noticed above, $\mathbb{V}_A(\_)$ defines a continuous presheaf, and we show below that it becomes a sheaf when $A$ does not have $K_1$ obstructions. For this we need a lemma (see \cite{APS}).

\begin{lemma}\label{pb}
Let $X$ be a one dimensional compact Hausdorff space and let $Y$, $Z\subseteq X$ be closed subsets of $X$. Let $A$ be a continuous field over $X$ without $K_1$ obstructions, and denote by $\pi^Z_Y\colon A(Y)\to A(Y\cap Z)$ 
and $\pi^Y_Z\colon A(Z)\to A(Y\cap Z)$ the quotient maps (given by restriction). Then, the map
\[\beta\colon V(A(Y)\oplus_{A(Y\cap Z)}A(Z))\to V(A(Y))\oplus_{V(A(Y\cap Z))}V(A(Z))\] defined by $\beta([(a,b)])=([a],[b])$ 
is an isomorphism.
\end{lemma}
\begin{proof}
We know from Corollary \ref{thm:sheaf} that $\Cu_A$ is a sheaf in this case. Thus the map 
\[
\Cu(A(Y)\oplus_{A(Y\cap Z)}A(Z))\to \Cu(A(Y))\oplus_{\Cu(A(Y\cap Z))}\Cu(A(Z))\,,
\]
given by $[(a,b)]\mapsto ([a],[b])$, is an isomorphism in $\Cu$, whence it maps compact elements to compact elements. Since $A(Y)\oplus_{A(Y\cap Z)}A(Z)$ is isomorphic to $A(Y\cup Z)$ and this algebra is stably finite (because all of its fibers have stable rank one), we have that the compact elements of $\Cu(A(Y)\oplus_{A(Y\cap Z)}A(Z))$ can be identified with $V(A(Y)\oplus_{A(Y\cap Z)}A(Z))$.

Using this identification, we have that $[(a,b)]$ in $\Cu(A(Y)\oplus_{A(Y\cap Z)}A(Z))$ is compact if and only if $[a]$ and $[b]$ are compact. On the other hand, if $[a]$ and $[b]$ are compact (in $\Cu(A(Y))$ and $\Cu(A(Z))$ respectively) and $[a]=[b]$ in $\Cu(A(Y\cap Z)$, then the pair $([a],[b])$ belongs to $V(A(Y))\oplus_{V(A(Y\cap Z))}V(A(Z))$, and every element of this pullback is obtained in this manner. The conclusion now follows easily.
\end{proof}

\begin{proposition}
Let $X$ be a one dimensional compact Hausdorff space and let $A$ be a continuous field over $X$ without $K_1$ obstructions.
Then, $$\begin{array}{cccc}
    \mathbb{V}_A(\_):& \mathcal{V}_X&\to& \text{Sg}\\
    &U&\mapsto& V(A(U))\,
    \end{array}$$ 
is a sheaf and the natural transformation
$\mathbb{V}_A(\_)\to \Gamma(\_, F_{V(A(\_))})$ is an isomorphism of sheaves.
\end{proposition}
\begin{proof}
Note that $\mathbb{V}_A(\_)$ is a sheaf thanks to Lemma \ref{pb}. On the other hand, the fact that $\mathbb{V}_A(\_)$ is isomorphic to the sheaf of continuous sections $\Gamma(\_, F_{V(A(\_))})$ follows from Theorem 2.2 in \cite{WE}.
\end{proof}

\begin{theorem}\label{V(A)}
Let $X$ be a compact Hausdorff space and let $A$ and $B$ be $\mathrm{C}(X)$-algebras such that all fibers have stable rank one. Consider the following conditions:
\begin{enumerate}[\rm(i)]
 \item $\Cu(A)\cong \Cu(B)$ preserving the action of $\Cu(\mathrm{C}(X))$,
\item $\Cu_A(\_)\cong \Cu_B(\_) $,
\item $\mathbb{V}_A(\_)\cong\mathbb{V}_B(\_) $.
\end{enumerate}
Then {\rm (i)} $\implies$ {\rm (ii)} $\implies$ {\rm (iii)}. If $X$ is one dimensional, then also {\rm (ii)} $\implies$ {\rm (i)}. If, furthermore, $A$ and $B$ are continuous fields without $K_1$ obstructions such that for all $x\in X$ the fibers $A_x$, $B_x $ have real rank zero, then {\rm (iii)} $\implies$ {\rm (ii)} and so all three conditions are equivalent.
\end{theorem}
\begin{proof}
We may assume that both $A$ and $B$ are stable.

(i) $\implies$ (ii):  Let $\varphi\colon\Cu(A)\to \Cu(B)$ be an isomorphism such that $\varphi(xy)=x\varphi(y)$, for any $x\in\Cu(\mathrm{C}(X))$ and $y\in \Cu(A)$. We need to verify that $\varphi(\Cu(\C_0(X\setminus V)A)\subseteq \Cu(\C_0(X\setminus V)B)$, whenever $V$ is a closed subset of $X$. This will entail that $\varphi$ induces a semigroup map $\varphi_V\colon\Cu(A(V))\to \Cu(B(V))$, which is an isomorphism as $\varphi$ is.

Let $[fa]\in\Cu(\C_0(X\setminus V)A)$, for $f\in \C_0(X\setminus V)_+$ and $a\in A_+$. Then, if $\varphi([a])=[b]$ for some $b\in B_+$, we have that $\varphi([fa])=[f]\varphi([a])=[f][b]=[fb]$, and $fb\in \C_0(X\setminus V)B$. Thus $\varphi(\Cu(\C_0(X\setminus V)A)\subseteq\varphi(\C_0(X\setminus V)B)$. The rest of the argument is routine.

(ii) $\implies$ (iii): Note that, as all fibers have stable rank one, $\Cu(A(U))$ (respectively, $\Cu(B(U))$) is a stably finite algebra for
 each closed subset $U$. In this case, $V(A(U))$ can be identified with the subset of compact elements of $\Cu(A(U))$. 
Therefore, the given isomorphism $\Cu_A(U)\cong\Cu_B(U)$ maps $\mathbb{V}_A(U)=V(A(U))$ injectively onto $\mathbb{V}_B(U)=V(B(U))$.

Now assume that $X$ is one dimensional, and let us prove that (ii) $\implies$ (i): The isomorphism of sheaves gives, in particular, an isomorphism $\varphi\colon \Cu(A)\to\Cu(B)$. We need to verify that $\varphi$ respects the action of $\Cu(\mathrm{C}(X))$. By \cite{Robert0}, $\Cu(\mathrm{C}(X))\cong \mathrm{Lsc}(X,\overline{\mathbb{N}})$.
In this case, any $f\in \mathrm{Lsc}(X,\overline{\mathbb{N}})$ may be written as: 
$$f=\sum^{\infty}_{i=1}\mathds{1}_{U_i} \text{ where } U_i=f^{-1}((i,\infty]).$$

Thus, in order to check that $\varphi(f[a])=f\varphi([a])$, it is enough to verify it when $f=\mathds{1}_U$, where $U$ is an open set of $X$.

Notice that $\mathds{1}_U[a]=[ga]$ where $g\in \mathrm{C}(X)_+$ has $\supp(g)=U$. Given $[a]\in \Cu(A)$ we denote by $\supp([a])=\{x\in X\mid \pi_x([a])\neq 0\}$. Observe that $\supp\varphi([a])=\supp([a])$, and that $\supp(\mathds{1}_U\varphi([a]))=U\cap \supp([a])$. 
Let $K\subseteq \supp(\mathds{1}_U\varphi([a]))=\supp(\varphi(\mathds{1}_U[a]))$ be a closed set. 
Then $\pi_K(\mathds{1}_U\varphi([a]))=\pi_K(\varphi(\mathds{1}_U[a]))$, where $\pi_K\colon A\to A(K)$ is the quotient map. Indeed, it follows from the commutative diagram 
\[
\xymatrix{
\Cu(A) \ar[r]^{\varphi}\ar[d]_{\pi_K} &  \Cu(B)\ar[d]^{\pi_K} \\                                                                         
\Cu(A(K)) \ar[r]^{\varphi_K} & \Cu(B(K))\ \, ,
}
\]
that $\pi_K(\mathds{1}_U[a])=\pi_K([ga])=\pi_K[a]$, since $g$ becomes invertible in $A(K)$. Therefore, 
\[
\pi_K(\varphi(\mathds{1}_U[a]))=\pi_K(\varphi([ga]))=\varphi_K \pi_K([ga])=\varphi_K\pi_K([a])\,.
\]
On the other hand, $\pi_K(\mathds{1}_U\varphi([a]))=\pi_K\varphi([a])=\varphi_K\pi_K([a])$.

Now write $[a]=\sup[a_n]$, where $([a_n])$ is a rapidly increasing sequence in $\Cu(A)$, and $\mathds{1}_U=\sup \mathds{1}_{V_n}$, where $(V_n)$ is a rapidly increasing sequence of open sets. Then $(\mathds{1}_{V_n}[a_n])$ is a rapidly increasing sequence with
$\mathds{1}_U[a]=\sup \mathds{1}_{V_n}[a_n]$ and $\mathds{1}_U\varphi([a])=\sup \mathds{1}_{V_n}\varphi([a])$. By \cite[Lemma 2.5]{ADPS} choose, for each $n$, a compact set $K_n$ such that
$$\supp(\mathds{1}_{V_n}[a_n])\subseteq K_n \subseteq \supp(\mathds{1}_{V_{n+1}}[a_{n+1}])\,.$$ Then $K_n\subseteq V_{n+1}\cap \supp([a_{n+1}])\subseteq V_{n+1}$.

By the above, $\pi_{K_n}(\mathds{1}_{V_{n+1}}\varphi([a_{n+1}]))= \pi_{K_n}\varphi(\mathds{1}_{V_{n+1}}[a_{n+1}])$, and thus:
\[\pi_{K_n}(\mathds{1}_{V_n}\varphi([a_n]))\leq \pi_{K_n}(\mathds{1}_{V_{n+1}}\varphi([a_{n+1}]))=\pi_{K_n}(\varphi(\mathds{1}_{V_{n+1}}[a_{n+1}]))\leq \pi_{K_n}(\varphi(\mathds{1}_U[a])),\]
and
\[\pi_{K_n}(\varphi(\mathds{1}_{V_n}[a_n]))\leq \pi_{K_n}(\varphi(\mathds{1}_{V_{n+1}}[a_{n+1}]))=\pi_{K_n}(\mathds{1}_{V_{n+1}}\varphi([a_{n+1}]))\leq \pi_{K_n}(\mathds{1}_U\varphi([a])).\]
Since $\supp (\mathds{1}_{V_k}[a_k])=\supp(\varphi(\mathds 1_{V_k}[a_k]))=\supp(\mathds 1_{V_k}\varphi([a_k]))$, we may apply Lemma 2.4 in \cite{ADPS} to obtain that $\mathds{1}_{V_{n}}\varphi([a_{n}])\leq \varphi(\mathds{1}_U[a]))$ 
and  $\varphi(\mathds{1}_{V_{n}}[a_{n}])\leq \mathds{1}_U\varphi([a]))$. Taking suprema in both inequalities we obtain $\mathds{1}_U\varphi([a])=\varphi(\mathds{1}_U([a]))$.

\rm(iii)$\implies$ \rm(ii): We assume now that both $A$ and $B$ are continuous fields without $K_1$ obstructions such that $A_x$ and $B_x$ have real rank zero for all $x$. Let $\varphi\colon \mathbb{V}_A(\_)\to\mathbb{V}_B(\_)$ be a sheaf isomorphism. This induces a semigroup isomorphism $\varphi_x\colon V(A_x)\to V(B_x)$ for each $x\in X$. As $A(U)$ is a stably finite algebra for any closed subset $U$ of $X$, we will identify $V(A(U))$ with its image in $\Cu(A(U))$ whenever convenient.

Since $A_x$ has real rank zero, $V(A_x)$ forms a dense subset of $\Cu(A_x)$ so we can uniquely define an isomorphism $\Cu(A_x)\to \Cu(B_x)$ in $\Cu$ which we will still denote by $\varphi_x$. This map is defined
by $\varphi_x(z)=\sup_n\varphi_x(z_n)$ where $z=\sup z_n$ and $z_n\in V(A_x)$ for all $n\geq 0$ (see, e.g. \cite{ABP}, \cite{CEI} for further details). Let us prove that the induced bijective map $\tilde\varphi\colon F_{\Cu(A)}\to F_{\Cu(B)}$ is continuous, and hence an homeomorphism. This will define an isomorphism of 
sheaves $\Gamma(-,F_{\Cu_A(-)})\cong \Gamma(-,F_{\Cu_B(-)})$ and then, using Theorem~\ref{prop:isoseccions}, it follows that $\Cu_A(\_)$ and $\Cu_B(\_)$ are isomorphic. 

Denote by $\pi_A\colon F_{\Cu(A)}\to X$ and $\pi_B\colon F_{\Cu(B)}\to X$ the natural maps. Let $U$ be an open set of $X$ and $s\in\Cu(B)$. We are to show that $\tilde{\varphi}^{-1}(U^{\gg}_s)$ is open in $F_{\Cu(A)}$. Let $z\in \tilde{\varphi}^{-1}(U^{\gg}_s)$,  and put $x=\pi_A(z)$, so that $z\in \Cu(A_x)$ for some $x\in U$. Since $\tilde\varphi(z)=\varphi_x(z)\in U^{\gg}_s$, there exists $s''\gg s$ such that $\hat{s''}(x)\ll \varphi_x(z)$. Choose $s'$ such that $s\ll s'\ll s''$.

As $\hat{s''}(x)\ll \varphi_x(z)$ there exists $z'\ll z'\in V(A_x)$ such that $\hat{s''}(x)\ll\varphi_x(z')$. Now we can find a closed subset $W'$ whose interior contains $x$, and an element $v\in V(A(W'))$ such that $\pi_x(v)=z'$. Note that $\pi_x\varphi_{W'}(v)=\varphi_x(z')$. Also, since $\hat{s'}(x)\ll\hat{s''}(x)\ll\varphi_x(z')=\widehat{\varphi_{W'}(v)}(x)$, we may use Lemma \ref{limit} to find $W\subseteq W'$ such that $x\in\mathring{W}$ and
\[
\hat{s'}(y)\ll\widehat{\varphi_{W'}(v)}(y)\text{ for all }y\in W\,.
\]
Let $t\in\Cu(A)$ be such that $\pi_W(t)=\pi_{W}^{W'}(v)$. We now claim that $\mathring{W}^{\gg}_t\subseteq{\tilde\varphi}^{-1}(U_s^{\gg})$. Let $w\in{\mathring{W}}^{\gg}_t$, and put $y=\pi_A(w)\in W$. There exists $t'\gg t$ such that $\hat{t'}(y)\ll w$, whence, applying $\tilde\varphi$ it follows that
\[
\tilde\varphi(w)\gg\tilde\varphi(\hat{t'}(y))\gg \tilde\varphi(\hat{t}(y))=\tilde\varphi(\pi_W^{W'}(v)(y))=\tilde\varphi(\pi_y(v))=\pi_y(\varphi_W(v))=\widehat{\varphi_{W'}(v)}(y)\gg \hat{s'}(y)\,,
\]
and this shows that $w\in{\tilde\varphi}^{-1}(U_s^{\gg})$.

\end{proof}
\begin{remark}
{\rm We remark that the implication (ii) $\implies $ (i) in Theorem \ref{V(A)} above holds whenever $\Cu(C(X))\cong\mathrm{Lsc}(X,\overline{\mathbb{N}})$. This is the case for spaces more general than being just one dimensional, see \cite{Robert0}.} 
\end{remark}
\begin{remark}
{\rm Our Theorem \ref{V(A)} above allows us to rephrase the classification result obtained in \cite{DEN}, by using a single invariant. Namely, let $A,B$ be separable unital continuous fields of $AF$-algebras over $[0,1]$, and let $\tilde{\phi}:\Cu(A)\to \Cu(B)$ be an isomorphism that preserves the action by $\mathrm{Lsc}([0,1],\overline{\mathbb{N}})$ and such that $\tilde{\phi}([1_A])=[1_B]$. Then $\tilde\phi$  lifts to an isomorphism $\phi\colon A\to B$ of continuous fields of C$^*$-algebras. }
\end{remark}

\end{document}